\documentclass[a4paper,11pt]{article}
\usepackage{xypic}
\usepackage{mathtools}
\usepackage{amsmath}
\usepackage{mathrsfs}
\usepackage{amssymb}
\usepackage{amsthm}
\usepackage{dsfont}
\usepackage{graphicx}
\usepackage{bmpsize}
\usepackage{times}
\usepackage{ textcomp }
\usepackage{color}
\usepackage{mathtools}
\usepackage{latexsym, empheq, fancybox}
\usepackage{mathrsfs}
\usepackage{exscale}
\usepackage{booktabs, array}
\usepackage[english]{babel}
\newcommand{\nc}{\newcommand}
\nc{\rnc}{\renewcommand}
\numberwithin{equation}{section}
\newtheorem{theorem}{Theorem}[section]

\newtheorem{lemma}[theorem]{Lemma}

\newtheorem{proposition}[theorem]{Proposition}

\newtheorem*{question*}{Questions}

\newtheorem{remark}[theorem]{Remark}

\usepackage{authblk}
\usepackage{tikz}
\usepackage{xypic}
\usetikzlibrary{calc}
\usetikzlibrary{decorations.pathreplacing}

\newcommand{\R}{\mathbb{R}}
\newcommand{\N}{\mathbb{N}}

\newcommand{\To}{\rightarrow}

\title{\Large\textbf{Mean Li--Yorke chaos and multifractal analysis on subshifts}
}

\author{
		Zijie Lin
		\thanks{School of Mathematical Sciences and Institute of Mathematics, Nanjing Normal University, Nanjing 210023, P. R. China
			(E-mail: zjlin137@126.com)},
		Ercai Chen
		\thanks{School of Mathematical Sciences and Institute of Mathematics, Nanjing Normal University, Nanjing 210023, P. R. China,
			and Center  of Nonlinear Science, Nanjing University, Nanjing 210093, P. R.China
			(E-mail: ecchen@njnu.edu.cn),
		},
		Xiaoyao Zhou
		\thanks{School of Mathematical Sciences and Institute of Mathematics, Nanjing Normal University, Nanjing 210023, P. R. China
			(E-mail: zhouxiaoyaodeyouxian@126.com)}
}

\begin{document}
	\date{}
	\maketitle
	
	\begin{abstract}
		In the present paper, we use the generalized multifractal framework introduced by Olsen to study the Bowen entropy and packing entropy of historic sets with typical weights over aperiodic and irreducible shifts of finite type.
		Following those results and a transfer from almost everywhere to everywhere, we show that for each point $\omega$ in a irreducible shift of finite type $\Sigma_A$, the Bowen entropy of the set consisting of all the points that are mean Li-Yorke pairs with $\omega$ is $0$, and its packing entropy is full. This result is beyond the ergodic theory. Also, by the transfer from almost everywhere to everywhere, we show that for each point $\omega$ in a irreducible shift of finite type $\Sigma_A$, the Bowen entropy of the set consisting of all the points that are Li-Yorke pairs with $\omega$ is full. This result is also beyond the ergodic theory.
	\end{abstract}
	
\section{Introduction}
	Recall that a \emph{topological dynamical system} is a pair $(X,T)$ where $X$ is a compact metric space with a metric $\rho$ and $T:X\To X$ is a continuous surjection. For $x\in X$ and $\epsilon>0$, let $B(x,\epsilon):=\{y\in X:\rho(x,y)<\epsilon\}$ for $x\in X$ and $B(C,\epsilon):=\{y\in X:\rho(y,C):=\inf_{z\in C}\rho(y,z)<\epsilon\}$ for $C\subset X$.

	Chaos, first introduced in \cite{LY75}, plays an important role in the study of complexity for dynamical systems.
	For a topological dynamical system $(X,T)$ and $x,y\in X$, recall that $(x,y)$ is a \emph{Li--Yorke scrambled pair} if
	$$\liminf_{n\To\infty}\rho(T^nx,T^ny)=0\text{ and }\limsup_{n\To\infty}\rho(T^nx,T^ny)>0.$$
	And $(X,T)$ is \emph{Li--Yorke chaotic} if there is an uncountable Li--Yorke scrambled set $S\subset X$, that is, $(x,y)$ is Li--Yorke scrambled pair for any $x\neq y\in S$.
	A pair $(x,y)\in X\times X$ is said to be a \emph{mean Li--Yorke pair} if
	$$
	\liminf_{n\rightarrow\infty}\frac{1}{n}\sum\limits_{i=0}^{n-1}\rho(T^ix,T^iy)=0\text{ and } \limsup_{n\rightarrow\infty}\frac{1}{n}\sum\limits_{i=0}^{n-1}\rho(T^ix,T^iy)>0.
	$$
	And $(X,T)$ is \emph{mean Li--Yorke chaotic} if there is an uncountable mean Li--Yorke scrambled set $S\subset X$, that is, $(x,y)$ is mean Li--Yorke scrambled pair for any $x\neq y\in S$.
	
	After Li--Yorke chaos, distributional chaos was first introduced in \cite{SS94} and was generalized in \cite{BSS05}, \cite{PS05} and \cite{PS06}.
	By describing the
	densities of trajectory approach time sets, distributional chaos reveals more rigorous complexity hidden in Li--Yorke chaos.
	
	We will now briefly review the definitions of three types of  distributional chaos.
	For ${x},{y}\in X$, define the lower distributional function $F_{{x},{y}}$ and  upper distributional function $F_{{x},{y}}^*$ from $(0,+\infty)$ to $[0,1]$ by
	\begin{equation}
	\begin{aligned}
	&F_{x, y}(\epsilon) =\liminf_{n \rightarrow \infty} \frac{1}{n} \#\left(\left\{0 \leq i<n : \rho\left(T^{i}x, T^{i}y\right)<\epsilon\right\}\right), \\
	&F_{x, y}^{*}(\epsilon) =\limsup _{n \rightarrow \infty} \frac{1}{n} \#\left(\left\{0 \leq i<n : \rho\left(T^{i}x, T^{i}y\right)<\epsilon\right\}\right),
	\end{aligned}
	\end{equation}
	where $\#(\cdot)$ denotes the cardinality of a set.
	A couple $({x},{y})\in X\times X$ is called a DC1 { pair} if
	\[
	F_{x, y}^{*}(\epsilon) \equiv 1 \text { on }(0,+\infty) \text { and } F_{x, y}(\epsilon) \equiv 0 \text { on some }\left(0, \epsilon_{0}\right],
	\]
	a DC2 { pair} if
	\[
	F_{x,y}^*(\epsilon)\equiv 1 \text { on }(0,+\infty) \text { and } F_{{x},{y}}(\epsilon)<1 \text { on some }\left(0,\epsilon_{0}\right],
	\]
	and a DC3 { pair} if
	\[
	F_{{x},{y}}(\epsilon)<F_{{x},{y}}^*(\epsilon) \text{ on some } \left(\epsilon_{0},\epsilon_{1}\right ].
	\]
	A set $C\subset X$ is said to be a $\rm{DC}i$ ($i=1,2$ or 3) { scrambled set} if each pair of different points in $C$ forms a $\rm {DC}i$ pair. In general, $f$ is said to be $\rm {DC}i$ chaotic if it has an uncountable $\rm{DC}i$ scrambled set.
	It is observed in \cite{Dow14} that DC2 chaos is equivalent to mean Li--Yorke chaos (see \cite{Dow14} for details), and a pair is a DC2 pair if and only if it is a mean Li--Yorke pair.
	In \cite{WCZ19}, the authors show the mean Li--Yorke chaos in some random dynamical systems.
	In \cite{Xiao21}, the author proves that mean Li-Yorke chaotic sets along polynomial sequence in $\beta$-transformation are full Hausdorff dimension.
	
	For $x\in X$, denote the set of all the points that are Li--Yorke pairs with $x$ by
	$$\mathrm{LY}_{x}(X,T)=\{y\in X:(x,y)\text{ is a Li--Yorke pair.}\},$$
	and the set of all the points that are mean Li--Yorke pairs with $x$ by
	$$\mathrm{ML}_{x}(X,T)=\{y\in X:(x,y)\text{ is a mean Li--Yorke pair.}\}.$$
	Then natural questions are
	\begin{question*}\label{q}
		Is the set $\mathrm{LY}_{x}(X,T)$( or $\mathrm{ML}_{x}(X,T)$) nonempty?
		If $\mathrm{LY}_{x}(X,T)$( or $\mathrm{ML}_{x}(X,T)$) is nonempty, how big is it?
	\end{question*}
	By the definition of mean Li-Yorke pair, the set $\mathrm{ML}_{x}(X,T)$ can be seen as a historic set of weighted Birkhoff averages, which motivates us use the tools of multifractal analysis.
	Following the works \cite{Bes34} and \cite{Egg49}, multifractal analysis is systematized in \cite{Pes97}.
	Originally, multifractal analysis investigates the measures and dimensions of the so-called level sets, i.e., the sets on which the Birkhoff average converge to given numbers.
	In \cite{BOS07}, \cite{Ols04} and \cite{Ols03}, multifractal analysis is applied to a wider family of saturated sets, including sets on which the Birkhoff averages diverge.
	There are many results on Birkhoff averages via multifractal analysis.
	Fan, Schmeling and Wu(\cite{FSW11}) completely answer the problem of multifractal analysis of
	multiple ergodic averages in the case of symbolic dynamics for functions of two variables depending on the first coordinate.
	In 2016, Fan, Schmeling and Wu(\cite{FSW16}) present a complete solution on multifractal analysis of the limit of some multiple ergodic averages on symbolic space.
	Recently, Fan(\cite{Fan21}) study the multifractal behavior of weighted ergodic averages on symbolic space.
	See more early works in \cite{FLW02,Ols06,Ols08,OW03}.
	In the present paper, we investigate the sets on which the weighted Birkhoff averages diverge and apply our results to mean Li-Yorke chaotic sets.
	
	In the present paper, we focus on subshifts, an important class of dynamical systems.
	Let $\mathcal{A}=\{1,...,K\}$ be a finite alphabet. Recall that a \emph{full shift} is a topological dynamical system $(\Sigma,\sigma)$ where
	$$\Sigma=\mathcal{A}^\N=\{x=x_0x_1x_2\cdots:x_i\in\mathcal{A}, i\in\N\}$$
	and $\sigma$ is the shift map, that is, $(\sigma x)_i=x_{i+1}$ for $i\in\N$.
	For $n>0$, denote by $\Sigma_n$ the set of $n$-length finite words, and let $\Sigma_*=\bigcup_{n>0}\Sigma_n$.
	For a word $W\in\Sigma_*$, denote $|W|$ be the length of $W$ and $[W]=\{x\in\Sigma:x|_{[0,|W|)}=W\}$ be the corresponding cylinder set, where $x|_{[a,b)}=x_ax_{a+1}\cdots x_{b-1}$ for any $a<b\in\N$. A compatible metric on $\Sigma$ is the metric
	$$\rho(x,y)=K^{-\min\{n\ge 0:x_n\neq y_n\}}$$
	for any different  $x=x_0x_1x_2\cdots,y=y_0y_1y_2\cdots\in\Sigma$ and $\rho(x,x)=0$.
	A \emph{subshifts} is a subsystem of a full shift.
	Let $A$ be a $K\times K$ matrix with entries 0,1.
	We say that $(\Sigma_A,\sigma)$ is a \emph{shift of finite type} if
	$$\Sigma_A=\{x=x_0x_1x_2\cdots\in\mathcal{A}^\N:A_{x_i,x_{i+1}}=1\text{ for each } i\in\N\}.$$
	Denote by $\Sigma_{A,n}$ the set of all the  admissible $n$-length words of $\Sigma_A$ and $\Sigma_{A,*}$ the set of all the admissible words of $\Sigma_A$.
	We say that $\Sigma_A$ is \emph{irreducible} if for any $i,j$, there exists $r>0$ such that $A^r_{i,j}>0$.
	We say that $\Sigma_A$ is \emph{aperiodic and irreducible} if there exists $r>0$ such that every entry of $A^r$ is strictly positive.
	
	Let $\nu$ be a $\sigma$-invariant ergodic measure on $\Sigma$. We call $\nu$ is \emph{quasi-Bernoulli} if there exists a constant $c>0$ such that for every $W,W'\in\Sigma_*$ with $WW'\in\Sigma_*$,
	$$c^{-1}\nu([W])\nu([W'])\le \nu([WW'])\le c\nu([W])\nu([W']).$$
	
	Then for an irreducible shift of finite type $(\Sigma_A,\sigma)$, we get the following theorems.
	
	\begin{theorem}\label{t:ly}
		Assume that $\Sigma_A$ is an irreducible shift of finite type. Then for any $\omega\in\Sigma_A$,
		$$h^B(\mathrm{LY}_\omega(\Sigma_A,\sigma))=h(\Sigma_A),$$
		that is,
		$$\dim_H(\mathrm{LY}_\omega(\Sigma_A,\sigma))=\dim_H(\Sigma_A).$$
	\end{theorem}
	
	\begin{theorem}\label{t:ml}
		Assume that $\Sigma_A$ is an irreducible shift of finite type. Then for any $\omega\in\Sigma_A$,
		$$h^B(\mathrm{ML}_\omega(\Sigma_A,\sigma))=0\text{ and }h^P(\mathrm{ML}_\omega(\Sigma_A,\sigma))=h(\Sigma_A),$$
		that is,
		$$\dim_H(\mathrm{ML}_\omega(\Sigma_A,\sigma))=0\text{ and }\dim_P(\mathrm{ML}_\omega(\Sigma_A,\sigma))=\dim_H(\Sigma_A).$$
	\end{theorem}

	In order to investigate the size of the sets consisting of mean Li--Yorke pair, we focus on the product of two full shifts. Let $\Omega=\Lambda^\N$ and $\Sigma=\mathcal{A}^\N$ be two full shifts, where $\mathcal{A}$ and $\Lambda$ are two finite alphabets.
	Define $\Gamma:=\Omega\times\Sigma$ and $\Gamma_A:=\Omega\times\Sigma_A$, where $\Sigma_A$ is an aperiodic and irreducible subshift of finite type with matrix $A$.
	For convenience, we denote the shift map by same notation $\sigma$ on different subshifts.
	Let $\Pi:\Gamma_A\To\Omega$ be the projection, that is, $\Pi(\omega,x)=\omega$.
	Denote by $\mathcal{M}_\nu(\Gamma)$(or $\mathcal{E}_\nu(\Gamma)$) the set of all the $\sigma$-invariant (ergodic) measures on $\Gamma$ with marginal $\nu$, that is, $\Pi_*(\mu)=\nu$.
	We define $\mathcal{M}_\nu(\Gamma_A)$ and $\mathcal{E}_\nu(\Gamma_A)$ on $\Gamma_A$ similarly.
	
	For a continuous function $f:\Gamma_A\To \R^d$, we say $f$ has \emph{bounded variation} if
	$$\sum_{i=0}^\infty\mathrm{var}_i f<\infty \text{ where }
	\mathrm{var}_i f:=\sup_{
		\substack{W\in\Gamma_{A,i}\\(\omega,x),(\omega',x')\in[W]}
	}|f(\omega,x)-f(\omega',x')|.\footnote{The symbol $|\cdot|$ denote the Euclidean norm of $\R^d$.}$$
	Let $$\mathcal{P}_A=\{\alpha\in\R^d:\text{there exists }\mu\in\mathcal{M}_\nu(\Gamma_A)\text{ such that }\int f d\mu=\alpha\}.$$
	Indeed, $\mathcal{P}_A$ is bounded by $\sup_{(\omega,x)\in\Gamma_A}|f(\omega,x)|$.
	Denote by $\mathcal{P}_A^o$ the interior of $\mathcal{P}_A$.	
	Given a continuous function $f:\Gamma_A\To \R^d$, for $\omega\in\Omega$ and $x\in\Sigma_A$, let
	$$S_n^mf(\omega,x):=\sum_{i=n}^{m-1}f(\sigma^i\omega,\sigma^ix)$$
	and denote by $A(\omega,x)$ the set of all the limit points of
	$\{\frac1nS_0^nf(\omega,x):n\in\N\}$.
	It is known that $A(\omega,x)$ is a compact and connected subset of $\R^d$.
	
	Recently, B\'ar\'any, Rams and Shi(\cite{BRS20}) investigate the Bowen entropy of the set $\{x\in\Sigma_A:A(\omega,x)=\{\alpha\}\}$.
	This notion can be seen as the notion of level sets in multifractal analysis. It is the notion on the multifractal framework introduced by Olsen. So we investigate the generalized multifractal framework(\cite{Ols03}, \cite{OW07}, \cite{Ols06} and \cite{Ols08}).
	For $C\subset\R^d$, let
	$$\Delta^\omega_{\mathrm{cap}}(C)=\{x\in\Sigma_A:A(\omega,x)\cap C\neq\emptyset\},$$
	$$\Delta^\omega_{\mathrm{sub}}(C)=\{x\in\Sigma_A:A(\omega,x)\subset C\}$$
	and
	$$\Delta^\omega_{\mathrm{sup}}(C)=\{x\in\Sigma_A:A(\omega,x)\supset C\}.$$
	Moreover, if $C$ is a compact and connected subset of $\R^d$, let
	$$\Delta^\omega_{\mathrm{equ}}(C)=\{x\in\Sigma_A:A(\omega,x)=C\}.$$
	Then we will show the Bowen entropy of those sets.
	\begin{theorem}\label{t:1}
		Let $\Sigma_A\subset\Sigma$ be an aperiodic and irreducible shift of finite type, and let $f:\Gamma_A\To \R^d$ be a continuous map with bounded variation. If $\nu$ is a quasi-Bernoulli $\sigma$-invariant ergodic measure on $\Omega$, then for any nonempty subset  $C\subset\mathcal{P}_A^o$ and $\nu$-almost everywhere $\omega\in\Omega$, we have following statements:
		\begin{itemize}
			\item[(1)] For $\Delta^\omega_{\mathrm{cap}}(C)$ and $\Delta^\omega_{\mathrm{sub}}(C)$,
			$$
			\begin{aligned}
			&h^B(\Delta^\omega_{\mathrm{cap}}(C))=h^B(\Delta^\omega_{\mathrm{sub}}(C))\\
			=&\sup\{h_\mu:\mu\in\mathcal{E}_\nu(\Gamma_A)\text{ and }\int f d\mu\in C\}-h_\nu\\
			=&\sup\{h_\mu:\mu\in\mathcal{M}_\nu(\Gamma_A)\text{ and }\int f d\mu\in C\}-h_\nu\\
			=&\sup_{\alpha\in C}\inf_{p\in\R^d}P_{\nu}(\langle p,f-\alpha\rangle);
			\end{aligned}
			$$
			\item[(2)] For $\Delta^\omega_{\mathrm{sup}}(C)$,
			$$
			\begin{aligned}
			h^B(\Delta^\omega_{\mathrm{sup}}(C))
			=&\inf_{\alpha\in C}\sup\{h_\mu:\mu\in\mathcal{E}_\nu(\Gamma_A)\text{ and }\int f d\mu=\alpha\}-h_\nu\\
			=&\inf_{\alpha\in C}\sup\{h_\mu:\mu\in\mathcal{M}_\nu(\Gamma_A)\text{ and }\int f d\mu=\alpha\}-h_\nu\\
			=&\inf_{\alpha\in C,p\in\R^d}P_{\nu}(\langle p,f-\alpha\rangle);
			\end{aligned}
			$$
			\item[(3)] If $C$ is compact and connected, then
			$$
			\begin{aligned}
			h^B(\Delta^\omega_{\mathrm{equ}}(C))
			=&\inf_{\alpha\in  C}\sup\{h_\mu:\mu\in\mathcal{E}_\nu(\Gamma_A)\text{ and }\int f d\mu=\alpha\}-h_\nu\\
			=&\inf_{\alpha\in C}\sup\{h_\mu:\mu\in\mathcal{M}_\nu(\Gamma_A)\text{ and }\int f d\mu=\alpha\}-h_\nu\\
			=&\inf_{\alpha\in C,p\in\R^d}P_{\nu}(\langle p,f-\alpha\rangle).
			\end{aligned}
			$$
		\end{itemize}
		
	\end{theorem}
	For the packing entropy, we will show for $\Delta^\omega_{\mathrm{sub}}(C)$ and $\Delta^\omega_{\mathrm{equ}}(C)$.
	\begin{theorem}\label{t:4}
	Let $\Sigma_A\subset \Sigma$ be an aperiodic and irreducible shift of finite type, and let $f:\Gamma_A\To \R^d$ be a continuous map with bounded variation. If $\nu$ is a quasi-Bernoulli $\sigma$-invariant ergodic measure on $\Omega$, then for any nonempty subset $C\subset\mathcal{P}_A^o$ and $\nu$-almost everywhere $\omega\in\Omega$, we have following statements:
	\begin{itemize}
		\item[(1)] For $\Delta^\omega_{\mathrm{sub}}(C)$,
		$$
		\begin{aligned}
		h^P(\Delta^\omega_{\mathrm{sub}}(C))
		=&\sup\{h_\mu:\mu\in\mathcal{E}_\nu(\Gamma_A)\text{ and }\int f d\mu\in C\}-h_\nu\\
		=&\sup\{h_\mu:\mu\in\mathcal{M}_\nu(\Gamma_A)\text{ and }\int f d\mu\in C\}-h_\nu\\
		=&\sup_{\alpha\in C}\inf_{p\in\R^d}P_{\nu}(\langle p,f-\alpha\rangle);
		\end{aligned}
		$$
		\item[(2)] If $C$ is compact and connected, then $$
		\begin{aligned}
		h^P(\Delta^\omega_{\mathrm{equ}}(C))
		=&\sup\{h_\mu:\mu\in\mathcal{E}_\nu(\Gamma_A)\text{ and }\int f d\mu\in C\}-h_\nu\\
		=&\sup\{h_\mu:\mu\in\mathcal{M}_\nu(\Gamma_A)\text{ and }\int f d\mu\in C\}-h_\nu\\
		=&\sup_{\alpha\in C}\inf_{p\in\R^d}P_{\nu}(\langle p,f-\alpha\rangle).
		\end{aligned}
		$$
	\end{itemize}
	\end{theorem}

	This paper is organized as follow. Section \ref{s:2} is an introduction of some notations and their properties. In Section \ref{s:3}, we focus on the Bowen entropy of $\Delta^\omega_{\mathrm{cap}}(C)$ and $\Delta^\omega_{\mathrm{sub}}(C)$. In Section \ref{s:4}, a basic construction is shown for the lower bound of Bowen entropy of $\Delta^\omega_{\mathrm{sup}}(C)$. Section \ref{s:5} is the calculation of Bowen entropy of $\Delta^\omega_{\mathrm{equ}}(C)$. In Section \ref{s:6}, we turn to the packing entropy of $\Delta^\omega_{\mathrm{sub}}(C)$ and $\Delta^\omega_{\mathrm{equ}}(C)$. In Section \ref{s:7}, based on our results, we prove Theorem \ref{t:ml}. Section \ref{s:8} is the proof of Theorem \ref{t:ly}.
	
\section{Preliminaries}\label{s:2}

For $Z\subset\Sigma=\mathcal{A}^\N$, we introduce three definitions of entropy for a subset (See details in \cite{FH12,ZCC12}).
Recall the upper capacity entropy of $Z$ is defined as follows.
For $\epsilon>0$, a subset $E\subset Z$ is called a \emph{$(n,\epsilon)$-separated set of $Z$} if $\rho_n(x,y)>\epsilon$ for any different points $x\neq y\in E$, and a \emph{$(n,\epsilon)$-spanning set of $Z$} if $\bigcup_{x\in E}\{y\in\Sigma:\rho_n(x,y)<\epsilon\}\supset Z$.
\footnote{We set $\rho_n(x,y):=\max_{0\le i<n}\rho(\sigma^ix,\sigma^iy)$. }
Let $s_n(Z,\epsilon)$ be the maximum cardinality of $(n,\epsilon)$-separated set of $Z$, and $r_n(Z,\epsilon)$ be the minimum cardinality of $(n,\epsilon)$-spanning set of $Z$. Define the \emph{upper capacity entropy of $Z$} is
$$h^{UC}(Z)=\lim_{\epsilon\To0}\limsup_{n\To\infty}\frac1n\log r_n(Z,\epsilon)=\lim_{\epsilon\To0}\limsup_{n\To\infty}\frac1n\log s_n(Z,\epsilon).$$

Recall that the Bowen topological entropy of $E$ is defined as follows. For $s\ge 0$ and $N\in\N$, define
$$\mathcal{H}^s_{N}(Z)=\inf_{\mathcal{W}}\sum_{W\in\mathcal{W}}e^{-s|W|},$$
where the infimum is taken over all covers $\mathcal{W}$ of $Z$, consisting of cylinders whose length is larger than $N$.
Define $\mathcal{H}^s(Z)=\lim_{N\To\infty}\mathcal{H}^s_{N}(Z)\in[0,+\infty]$,
and the \emph{Bowen topological entropy} of $Z$
$$h^B(Z)=\inf\{s\ge 0:\mathcal{H}^s(Z)=0\}.$$

Note that $\dim_H(Z)=\frac{h^B(Z)}{\log K}$.
The upper bound of $h^B(Z)$ is given by
$$h^B(Z)\le \liminf_{n\To\infty}\frac1n\log\#\{W\in\Sigma_n:[W]\cap Z\neq\emptyset\}.$$
The lower bound can get by a version of Frostman Lemma as follows.
\begin{lemma}\label{l:FL}\cite[Lemma 3.1]{BRS20}
	Let $Z\subset\Sigma$ and $s\ge 0$. Suppose that there exists a probabilistic measure $\mu$ on $Z$ satisfying that there exists a constant $C>0$ such that for every cylinder $W$, we have $\mu([W]\cap Z)\le Ce^{-s|W|}$. Then $h^B(Z)\ge s.$
\end{lemma}

Recall that the packing entropy of a subset $Z\subset \Sigma$ is defined as follows.
For $s\ge0$ and $N\in\N$,
$$\mathcal{P}_N^s(Z)=\sup_{\mathcal{W}}\sum_{W\in\mathcal{W}}e^{-s|W|},$$
where the supremum is taken over all families $\mathcal{W}$,  consisting of pairwise disjoint cylinders satisfying that $[W]\cap Z\neq\emptyset$ and $|W|\ge N$ for any $W\in\mathcal{W}$.
Define $\mathcal{P}^s_*(Z)=\lim_{N\To\infty}\mathcal{P}^s_N(Z)\in[0,\infty]$, and
$$\mathcal{P}^s(Z)=\inf\left\{\sum_{i\in\N}\mathcal{P}^s_*(Z_i):\bigcup_{i\in\N}Z_i\supset Z\right\}.$$
The \emph{packing entropy of $Z$} is defined as
$$h^P(Z)=\inf\{s\ge 0:\mathcal{P}^s(Z)=0\}.$$
The lower bound of $h^P(Z)$ can get by a version of distribution principle as follows.
\begin{lemma}\label{l:DP}\cite[Proposition 2.3]{ZCC12}
	Let $Z\subset\Sigma$ and $s\ge 0$. Suppose that there exists a probabilistic measure $\mu$ on $Z$ satisfying that there exists a constant $C>0$ and a sequence $\{n_i\}\nearrow\infty$ such that for any $x\in Z$, we have $\mu([x|_{[0,n_i)}]\cap Z)\le Ce^{-sn_i}$. Then $h^P(Z)\ge s.$
\end{lemma}

It is well known that $h^B(Z)\le h^P(Z)\le h^{UC}(Z)$.

Recall that $\Pi:\Gamma\To\Omega$ is the projection with $\Pi(\omega,x)=\omega$.
Let $\mu$ be an ergodic $\sigma$-invariant measure on $\Gamma$.
Also, $\Pi_*\mu$ is an ergodic $\sigma$-invariant measure on $\Omega$.
By Shannon-McMillan-Breiman's Theorem,
$$h_\mu=\lim_{n\To\infty}-\frac1n\log\mu([(\omega,x)|_{[0,n)}])\text{ for $\mu$-a.e. }(\omega, x)\in\Gamma,$$
$$h_{\Pi_*\mu}=\lim_{n\To\infty}-\frac1n\log\Pi_*\mu([\omega|_{[0,n)}])\text{ for $\Pi_*\mu$-a.e. }\omega\in\Omega.$$

Denote by $\xi:=\{\{\omega\}\times\Sigma:\omega\in\Omega\}$ the partition of $\Gamma$ generated by $\xi(\omega):=\Pi^{-1}\{\omega\}$.
By Rohlin's Disintegration Theorem, denote the disintegration of $\mu$ by $\mu=\int\mu_\omega^\xi d\Pi_*\mu$. Since $\mu_\omega^\xi$ is supported on $\xi(\omega)$, it can be seen as a measure on $\Sigma$.
Define the conditional entropy of $\mu_\omega^\xi$ by
$$h_\mu^\xi:=\int-\log\mu_\omega^\xi([x_0])d \mu(\omega,x).$$
The following result is the corresponding version of Pinsker's formula \cite{Roh67}.
\begin{proposition}\cite[Theorem 3.2]{BRS20}
	If $\mu$ is an ergodic $\sigma$-invariant measure, then for $\Pi_*\mu$-a.e. $\omega\in\Omega$,
	$$h_\mu^\xi=\lim_{n\To\infty}-\frac1n\log\mu_\omega^\xi([x|_{[0,n)}])\text{ for $\mu_\omega^\xi$-a.e. } x\in\Sigma.$$
	Moreover,
	$$h_\mu=h_{\Pi_*\mu}+h_\mu^\xi.$$
\end{proposition}

Fix a continuous function $f:\Gamma_A\To\R$ and a quasi-Bernoulli $\sigma$-invariant ergodic measure $\nu$ on $\Omega$.
For $\omega\in\Omega$, let $$P(f,\omega):=\limsup_{n\To\infty}\frac1n\log Z_n(f,\omega).$$
where $Z_n(f,\omega):=\sum_{W\in\Sigma_{A,n}}\sup_{x\in [W]} \exp(S_0^nf(\omega,x))$.
Define the \emph{conditional pressure} $P_\nu(f)$ for $\nu$ by
$$P_\nu(f)=\int_{\Omega}P(f,\omega) d\nu(\omega),$$
It is proved in \cite{BRS20} that $P_\nu(f)=P(f,\omega)$ for $\nu$-a.e. $\omega\in\Omega$.
The following theorem is shown by Ledrappier and Walters \cite{LW77}.
\begin{theorem}\label{t:variation principle}
	Let $\nu$ be a $\sigma$-invariant measure on $\Omega$ and let $f:\Gamma_A\To\R$ be a continuous potential. Then
	$$P_\nu(f)=\sup\{h_\mu^\xi+\int f d\mu:\mu\in\mathcal{M}_\nu(\Gamma_A)\}.$$
\end{theorem}

In \cite[Theorem 2.3]{BRS20}, they prove the Bowen entropy of $\Delta^\omega_{\mathrm{equ}}(\{\alpha\})$ for $\alpha\in\mathcal{P}_A^o$.
\begin{theorem}[{\cite[Theorem 2.3]{BRS20}}]\label{t:fo}
	Let $\Sigma_A\subset \Sigma$ be an aperiodic and irreducible shift of finite type, and let $f:\Gamma_A\To \R^d$ be a continuous map with bounded variation. If $\nu$ is a quasi-Bernoulli $\sigma$-invariant ergodic measure on $\Omega$, then for any $\alpha\in\mathcal{P}_A^o$ and $\nu$-almost everywhere $\omega\in\Omega$,
	$$
	\begin{aligned}
	h^B(\Delta^\omega_{\mathrm{equ}}(\{\alpha\}))
	=&\sup\{h_\mu:\mu\in\mathcal{E}_\nu(\Gamma_A)\text{ and }\int f d\mu=\alpha\}-h_\nu\\
	=&\sup\{h_\mu:\mu\in\mathcal{M}_\nu(\Gamma_A)\text{ and }\int f d\mu=\alpha\}-h_\nu\\
	=&\inf_{p\in\R^d}P_{\nu}(\langle p,f-\alpha\rangle).
	\end{aligned}
	$$
\end{theorem}

\section{Bowen entropy of $\Delta^\omega_{\mathrm{cap}}(C)$ and $\Delta^\omega_{\mathrm{sub}}(C)$}\label{s:3}
In this section, we fix some $C\subset\mathcal{P}_A^o$, a continuous function $f:\Gamma_A\To\R^d$ with bounded variation and a quasi-Bernoulli $\sigma$-invariant ergodic measure $\nu$ on $\Omega$.
For $\alpha,p\in\R^d$, let
$$f^\alpha_p=\langle p,f-\alpha\rangle.$$

First, we show the upper bound of entropy of $\Delta^\omega_{\mathrm{cap}}(C)$.
For any $a\in\R^d$ and $B\subset\R^d$, let $\mathrm{dist}(a,B)=\inf_{b\in B}|a-b|$.
We need the following lemma, which is contained in \cite[Lemma 5.7]{BRS20}.

\begin{lemma}\label{l:p's_bound}
	For any $\alpha\in\mathcal{P}_A^o$ and $p\in\R^d$, let $\eta<\mathrm{dist}(\alpha,\R^d\setminus\mathcal{P}_A^o)$. Then we have $|p|\le \frac{P_\nu(f_p^\alpha)}{\eta}$.
\end{lemma}

\begin{proof}
	Since $\eta<\mathrm{dist}(\alpha,\R^d\setminus\mathcal{P}_A^o)$, $\alpha+\eta\cdot\frac{p}{|p|}\in\mathcal{P}_A$. Then there exists $\mu\in\mathcal{M}_\nu(\Gamma_A)$ such that $\int f d\mu=\alpha+\eta\cdot\frac{p}{|p|}$. Then by Theorem \ref{t:variation principle},
	$$P_\nu(f_p^\alpha)\ge h_\mu^\xi+\int f_p^\alpha d\mu\ge \eta|p|.$$
\end{proof}

By the above lemma, we can get the upper bound of $h^B(\Delta^\omega_{\mathrm{cap}}(C))$.

\begin{lemma}\label{l:1upperbound}
	For $\nu$-almost everywhere $\omega\in\Omega$,
	$$h^B(\Delta^\omega_{\mathrm{cap}}(C))\le \sup_{\alpha\in C}\inf_{p\in\R^d}P_\nu(f^\alpha_p).$$
\end{lemma}

\begin{proof}
	First, we assume that there exists $\eta>0$ such that $B(C,2\eta)\subset\mathcal{P}_A^o$. Let $\{\alpha_1,\alpha_2,...\}\subset C$ be a countable dense subset of $C$.
	
	Fix any $s>s_0>\sup_{\alpha\in C}\inf_{p\in\R^d}P_\nu(f^\alpha_p)$.
	Then for each $k$, there exists $p_k\in\R^d$ such that $P_\nu(f^{\alpha_k}_{p_k})<s_0$. By Lemma \ref{l:p's_bound}, we have $|p_k|\le \frac{s_0}\eta$.
	Fix $M\ge 1$ such that $\frac{3s_0}{M\eta}<\frac{s-s_0}2$.
	Since $C\subset\mathcal{P}_A$ and $\mathcal{P}_A$ is bounded, then there exists $k_M$ such that $\bigcup_{k=1}^{k_M}B(\alpha_k,\frac1M)\supset C$.
	For simplicity, we set $f_k:=f^{\alpha_k}_{p_k}$.
	Then there exists $\Omega_0\subset \Omega$ with $\nu(\Omega_0)=1$ such that for any $\omega\in\Omega_0$ and $k>0$, $P(f_k,\omega)=P_\nu(f_k)<s_0$.
	
	Fix $\omega\in\Omega_0$.
	Then there exists $N_1>0$ such that for every $n>N_1$ and $1\le k\le k_M$,
	$$\sum_{W\in\Sigma_{A,n}}\sup_{x\in [W]}e^{\langle p_k, S_0^nf(\omega,x)-n\alpha_k\rangle}<e^{s_0n}.$$
	By uniform continuity of $f(\omega,\cdot)$, there exists $N_2$ such that for every $n>N_2$, $W\in\Sigma_{A,n}$ and $x,y\in[W]$,
	$$\left|S_0^nf(\omega,x)-S_0^nf(\omega,y)\right|<\frac nM.$$
	Let $N>\max\{N_1,N_2\}$. By definition of $\Delta^\omega_{\mathrm{cap}}(C)$,
	$$\Delta^\omega_{\mathrm{cap}}(C)\subset\bigcup_{n\ge N}\bigcup_{k=1}^{k_M}\left\{[W]:W\in \Sigma_{A,n}\text{ and }\sup_{y\in[W]}\left|\frac1nS_0^nf(\omega,y)-\alpha_k\right|<\frac3M\right\}.$$
	For $n\ge N$ and $1\le k\le k_M$, let
	$$L_{n,k}=\left\{W\in \Sigma_{A,n}:\sup_{y\in[W]}\left|\frac1nS_0^nf(\omega,y)-\alpha_k\right|<\frac3M\right\}$$
	Hence,
	$$
	\begin{aligned}
		\mathcal{H}_N^s(\Delta^\omega_{\mathrm{cap}}(C))\le&\sum_{n\ge N}\sum_{k=1}^{k_M}\sum_{W\in L_{n,k}}e^{-ns}\\
		\le&\sum_{n\ge N}e^{-\frac{n(s-s_0)}2}\sum_{k=1}^{k_M}\sum_{W\in L_{n,k}}e^{-ns_0}\cdot e^{-|p_k|\cdot\frac{3n}M}\\
		\le&\sum_{n\ge N}e^{-\frac{n(s-s_0)}2}\sum_{k=1}^{k_M}\sum_{W\in L_{n,k}}e^{-ns_0}\cdot \sup_{y\in[W]}e^{\langle p_k,S_0^nf(\omega,y)-n\alpha_k\rangle}\\
		\le&\sum_{n\ge N}k_Me^{-\frac{n(s-s_0)}2}\To0\text{ as }N\To\infty.
	\end{aligned}
	$$
	Although the set $\Omega_0$ may depend on $s,s_0$, we can choose $\Omega'_0$ does not depend on $s,s_0$ by choosing $s,s_0\in\mathbb{Q}$ and the intersection of the corresponding $\Omega_0$.
	
	For general $C\subset\mathcal{P}_A^o$, let $C_n=\{
	\alpha\in C:\mathrm{dist}(\alpha,\R^d\setminus\mathcal{P}_A^o)>\frac1n\}$. Then we note that $C=\bigcup_{n\ge1}C_n$ and
	$$\Delta^\omega_{\mathrm{cap}}(C)=\bigcup_{n\ge1}\Delta^\omega_{\mathrm{cap}}(C_n).$$
	Thus
	$$h^B(\Delta^\omega_{\mathrm{cap}}(C))\le\sup_{n\ge 1}\sup_{\alpha\in C_n}\inf_{p\in\R^d}P_\nu(f^\alpha_p)=\sup_{\alpha\in C}\inf_{p\in\R^d}P_\nu(f^\alpha_p).$$
\end{proof}

Next, the lower bound of Bowen entropy of $h^B(\Delta^\omega_{\mathrm{sub}}(C))$ can be concluded from Theorem \ref{t:fo}, since
$$h^B(\Delta^\omega_{\mathrm{sub}}(C))\ge \sup_{\alpha\in C}h^B(\Delta^\omega_{\mathrm{equ}}(\{\alpha\})).$$
\section{Bowen entropy of $\Delta^\omega_{\mathrm{sup}}(C)$}\label{s:4}
In this section, we fix some $C\subset\mathcal{P}_A^o$, a continuous function $f:\Gamma_A\To\R^d$ with bounded variation and a quasi-Bernoulli $\sigma$-invariant ergodic measure $\nu$ on $\Omega$.

First, we give the upper bound of $h^B(\Delta^\omega_{\mathrm{sup}}(C))$.

\begin{lemma}\label{l:2upperbound}
	For any $\omega\in\Omega$,
	$$h^B(\Delta^\omega_{\mathrm{sup}}(C))\le \inf_{\alpha\in C,p\in\R^d}P(f^\alpha_p,\omega).$$
\end{lemma}

\begin{proof}
	Fix any $s>s_0>\inf_{\alpha\in C,p\in\R^d}P(f^\alpha_p,\omega)$.
	Then there exists $\alpha\in C$ and $p\in\R^d$ such that $P(f^\alpha_p,\omega)<s_0$.
	Then there exists $N_1>0$ such that for every $n>N_1$,
	$$\sum_{W\in\Sigma_{A,n}}\sup_{x\in [W]}e^{\langle p, S_0^nf(\omega,x)-n\alpha\rangle}<e^{s_0n}.$$
	Fix $M\ge 1$ such that $\frac{2|p|}{M}<\frac{s-s_0}2$.
	By uniform continuity of $f(\omega,\cdot)$, there exists $N_2$ such that for every $n>N_2$, $W\in\Sigma_{A,n}$ and $x,y\in[W]$,
	$$\left|S_0^nf(\omega,x)-S_0^nf(\omega,y)\right|<\frac nM.$$
	Let $N>\max\{N_1,N_2\}$.
	By definition of $\Delta^\omega_{\mathrm{sup}}(C)$,
	$$\Delta^\omega_{\mathrm{sup}}(C)\subset\bigcup_{n\ge N}\left\{[W]:W\in \Sigma_{A,n}\text{ and }\sup_{y\in[W]}\left|\frac1nS_0^nf(\omega,y)-\alpha\right|<\frac2M\right\}.$$
	Hence,
	$$
	\begin{aligned}
	\mathcal{H}_N^s(\Delta^\omega_{\mathrm{sup}}(C))\le&\sum_{n\ge N}\sum_{\substack{W\in\Sigma_{A,n}\\\sup_{y\in[W]}\left|S_0^nf(\omega,y)-n\alpha\right|<\frac{2n}M}}e^{-ns}\\
	\le&\sum_{n\ge N}e^{-\frac{n(s-s_0)}2}\sum_{\substack{W\in\Sigma_{A,n}\\\sup_{y\in[W]}\left|S_0^nf(\omega,y)-n\alpha\right|<\frac{2n}M}}e^{-ns_0}\cdot e^{-|p|\cdot\frac{2n}M}\\
	\le&\sum_{n\ge N}e^{-\frac{n(s-s_0)}2}\sum_{\substack{W\in\Sigma_{A,n}\\\sup_{y\in[W]}\left|S_0^nf(\omega,y)-n\alpha\right|<\frac{2n}M}}e^{-ns_0}\cdot \sup_{y\in[W]}e^{\langle p,S_0^nf_p^\alpha(\omega,y)-n\alpha\rangle}\\
	\le&\sum_{n\ge N}e^{-\frac{n(s-s_0)}2}\To0\text{ as }N\To\infty.
	\end{aligned}
	$$
	Thus, $h^B(\Delta^\omega_{\mathrm{sup}}(C))\le s$, and by the arbitrariness of $s$, it ends the proof.
\end{proof}

To give the lower bound, we construct a subset of $\Delta^\omega_{\mathrm{sup}}(C)$.
Here, we fix some countable subset $C'=\{\alpha_1,\alpha_2,\dots\}\subset C$ such that $\overline{C'}\supset C$.

For convenience, we give some notations and lemmas.

Although $f$ is defined on $\Gamma_A$, it can be extended to $\Gamma$ by the following way.
For each $x\in\Sigma$, let $n(x):=\max\{n\ge 0: x|_{[0,n)}\in\Sigma_{A,*}\}$ and $$f(\omega,x):=\sup_{y\in [x|_{[0,n(x))}]}f(\omega,y)$$
by the partial order $\leqq$ is defined as
$$\alpha\leqq\beta\overset{\mathrm{def}}{\Leftrightarrow}|\alpha|<|\beta|\text{ or }|\alpha|=|\beta|,\alpha_n<\beta_n\text{ where }n=\min\{1\le k\le d:\alpha_k\neq\beta_k\}.$$
For a subset $M\subset \Sigma$, we denote by $W_a^b(M)=\{W\in \mathcal{A}^{b-a}:\text{there exists }x\in M\text{ such that }x|_{[a,b)}=W\}$.

The following lemma is a direct generalization of \cite[Lemma 4.1]{BRS20}.

\begin{lemma}
	For $\epsilon>0$, $N\in\N$, $\alpha\in\R^d$ and $(\omega,x)\in\Gamma$. If $|S_0^nf(\omega,x)-\alpha|<\epsilon$ for all $n> N$, then for any $n>m>N$, we have $$|S_m^nf(\omega,x)-\alpha|<\frac{(n+m)\epsilon}{n-m}.$$
\end{lemma}

The following lemma is a corollary of \cite[Lemma 4.2]{BRS20} by taking $\phi_i=f(\sigma^i\omega,\cdot)$.

\begin{lemma}
	Let $\{q_j\}_{j\in\N}$ be an increasing sequence of integers satisfying $\lim_{j\To\infty}\frac{q_j}{j}=\infty$, $q_{j+1}-q_j>2r$ and $\lim_{j\To\infty}(q_{j+1}-q_j)=\infty$. Let $\pi:\Sigma\To\Sigma$ be a map satisfying the following properties:
	\begin{itemize}
		\item[(1)] if $x|_{[0,n)}=y|_{[0,n)}$ for $q_j<n\le q_{j+1}$, then $(\pi x)|_{[0,q_j)}=(\pi y)|_{[0,q_j)}$;
		\item[(2)] if $x_k\neq(\pi x)_k$, then $k\in\{q_j+1,\cdots,q_j+r\}$ for some $j$.
	\end{itemize}

	Then for any $(\omega,x)\in\Gamma$, $A(\omega,x)=A(\omega,\pi x)$ and for any $X\subset\Sigma$, $h^B(X)=h^B(\pi(X))$.
\end{lemma}

Now we construct a suitable subset of $\Sigma$.

\begin{lemma}\label{l:Moren}
	Fix any $\delta>0$. There exists $\widetilde{M}\subset\Sigma_A$ such that
	\begin{itemize}
		\item[(1)] for any $x\in \widetilde{M}$, $A(\omega,x)\supset C$;
		\item[(2)] $h^B(\widetilde{M})\ge \inf_{\alpha\in C}\sup\{h_\mu:\mu\in\mathcal{E}_\nu(\Gamma_A)\text{ and }\int f d\mu=\alpha\}-h_\nu-2\delta$.
	\end{itemize}
\end{lemma}

\begin{proof}
	For each $k\in\N$, we choose $\mu_k\in\mathcal{E}_\nu(\Gamma_A)$ such that $\int f d\mu_k=\alpha_k$ and
	$$h_{\mu_k}>\sup\{h_\mu:\mu\in\mathcal{E}_\nu(\Gamma_A)\text{ and }\int f d\mu=\alpha_k\}-\delta.$$
	By Rohlin's Disintegration Theorem, denote the disintegration of $\mu_k$ by $$\mu_k=\int\mu_k^{\omega,\xi}d\nu(\omega).$$
	By ergodic theorem, there exists $G_k\subset \Gamma_A$ such that $\mu_k(G_k)=1$ and
	$$\lim_{n\To\infty}\frac1nS_0^nf(\omega,x)=\alpha_k\text{ for each }(\omega,x)\in G_k.$$
	Therefore, let $$G'_k=\{\omega\in\Omega:\mu_k^{\omega,\xi}(\Delta^\omega_{\mathrm{equ}}(\{\alpha_k\}))=1\},$$
	and then we have $\nu(G'_k)=1$. Let $G^*=\bigcap_{k\in\N}G'_k$. Then we have $\nu(G^*)=1$ and for any $k\in\N$ and $\omega\in G^*$, we have $\mu_k^{\omega,\xi}(\Delta^\omega_{\mathrm{equ}}(\{\alpha_k\}))=1$.
	On the other hand, there exists $G^{**}\subset\Omega$ such that $\nu(G^{**})=1$ and for any $\omega\in G^{**}$ and $k\in\N$,
	$$\lim_{n\To\infty}-\frac1n\log\mu_k^{\omega,\xi}([x|_{[0,n)}])=h_{\mu_k}-h_\nu\text{ for $\mu_k^{\omega,\xi}$-a.e. $x\in\Sigma_A$}.$$
	
	Let $G=G^*\cap G^{**}$. Now we fix $\omega\in G$ and a sequence $\{\epsilon_L\}_{L\in\N}$ with $\lim_{L\To\infty}\epsilon_L=0$. For each  $k\in\N$, there exists $M_k\subset\Sigma_A$ such that $\mu_k^{\omega,\xi}(M_k)>1-\delta$, and for each $\epsilon_L$ there exists $N_{L,k}$ such that for any $x\in M_k$ and $n>N_{L,k}$,
	$$\left|\frac1nS_0^nf(\omega,x)-\alpha_k\right|<\epsilon_L,$$
	and
	$$\left|-\frac1n\log\mu_k^{\omega,\xi}([x|_{[0,n)}])-(h_{\mu_k}-h_\nu)\right|<\epsilon_L.$$
	
	Next, for each $k\in\N$, we claim that for any $N\in\N$, there exists $n_0>N$ such that for any $n>n_0$,
	$$\log\#W_{n_0}^n(M_k)>(n-n_0)(h_{\mu_k}-h_\nu-\delta).$$
	We prove it by a contradiction. Assume that there exists $N$ such that for any $n>N$, there is $n'>n$ such that $\log\#W_{n}^{n'}(M_k)\le(n'-n)(h_{\mu_k}-h_\nu-\delta)$.
	Then there is a sequence $\{n_i\}_{i\ge 1}$ such that for each $i\ge 1$,  $\log\#W_{n_i}^{n_{i+1}}(M_k)\le(n_{i+1}-n_{i})(h_{\mu_k}-h_\nu-\delta)$.
	Thus, $\log\#W_0^{n_{i+1}}(M_k)\le \log\#W_0^{n_1}(M_k)+(n_{i+1}-n_1)(h_{\mu_k}-h_\nu-\delta)$, that is,
	$$\limsup_{i\To\infty}\frac1{n_{i+1}}\log\#W_0^{n_{i+1}}(M_k)\le h_{\mu_k}-h_\nu-\delta.$$
	For $i$ large enough and any $x\in M_k$, we have
	$$-\frac1{n_{i+1}}\log\mu_k^{\omega,\xi}([x|_{[0,n_{i+1})}])>h_{\mu_k}-h_\nu-\frac{\delta}{2},$$
	that is, $\mu_k^{\omega,\xi}([x|_{[0,n_{i+1})}])<\exp(-n_{i+1}(h_{\mu_k}-h_\nu-\frac{\delta}{2}))$.
	Since $\mu_k^{\omega,\xi}(M_k)>1-\delta$, we have $\#W_0^{n_{i+1}}(M_k)\ge (1-\delta)\exp(n_{i+1}(h_{\mu_k}-h_\nu-\frac{\delta}{2}))$.
	So
	$$\liminf_{i\To\infty}\frac1{n_{i+1}}\log\#W_0^{n_{i+1}}(M_k)\ge h_{\mu_k}-h_\nu-\frac\delta2, $$
	which is a contradiction.
	
	So for each $k\in\N$, we choose a sequence $\{n_i^k\}_{i\ge 1}$ such that for any $i\ge 1$ and $n>n_i^k$, we have
	$$\log\#W_{n_i^k}^n(M_k)>(n-n_i^k)(h_{\mu_k}-h_\nu-\delta).$$
	
	For each $L\in\N$, by the denseness of $\{\alpha_1, \alpha_2,\dots\}$, there exists $J(L)$ such that
	$$\bigcup_{k=1}^{J(L)}B(\alpha_k,\epsilon_L)\supset C.$$
	Define $I=\{(L,k):L\ge 1,\,1\le k\le J(L)\}$ and a map $\lambda:I\To\{1,2,\dots\}$ by $\lambda(L,k)=k+\sum_{i=1}^{L-1}J(i)$.
	It is easy to see that $\lambda$ is a bijection.
	Define $\{\epsilon'_j\}_{j\in\N}$, $\{M'_j\}_{j\in\N}$, $\{\mu'_j\}_{j\in\N}$ and $\{\alpha'_j\}$ by $\epsilon'_j=\epsilon_L$, $M'_j=M_k$, $\mu'_j=\mu_k$ and $\alpha'_j=\alpha_k$ where $j=\lambda(L,k)$.
	So $\lim_{j\To\infty}\epsilon'_j=0$.
	Now we define a sequence $\{T_j\}_{j\in\N}$ satisfying the following properties:
	\begin{itemize}
		\item $T_0=0$, $T_1>N_{1,1}$;
		\item $T_j>N_{L,k}$ and $T_j\in\{n_i^k\}_{i\in\N}$, where $j+1=\lambda(L,k)$;
		\item $\lim_{j\To\infty}\frac{\sum_{i=1}^{j}T_i}{T_{j+1}}=0$;
		\item $\log\#W_0^n(M'_j)\ge n(h_{\mu'_j}-h_\nu-\delta)$ for all $n>T_j$;
		\item $\lim_{j\To\infty}T_j\epsilon'_j=\infty$.
	\end{itemize}
	For $j\ge 2$, define $1<r_j\le 2$ and $m(j)\in\N$ such that
	$$r_j^{m(j)}=\frac{T_j}{T_{j-1}},\,\lim_{j\To\infty}r_j=1,\,\lim_{j\To\infty}\frac{r_j-1}{\epsilon'_j}=\infty.$$
	For $j\ge 2$, let $t_l^j=\lfloor (r_j)^lT_{j-1}\rfloor$ for $l=0,1,\dots,m(j)$.
	Notice that $t_0^j=T_{j-1}$ and $t_{m(j)}^j=T_j=t_0^{j+1}$.
	Define a set
	$$
	\begin{aligned}
	M=\{x\in\Sigma:&x|_{[0,T_1)}\in W_0^{T_1}(M'_1),\\ &x|_{[t_l^j,t_{l+1}^j)}\in W_{t_l^j}^{t_{l+1}^j}(M'_j)\text{ for }j\ge 2\text{ and }0\le l\le m(j)-1\}.
	\end{aligned}
	$$
	Now we will show that the set $M$ is required.

	To prove $(1)$, fix any $\alpha_k\in C'$ and $x\in M$. We claim that
	$$\lim_{L\To\infty}\left|\frac1{T_{\lambda(L,k)}}S_0^{T_{\lambda(L,k)}}f(\omega,x)-\alpha_k\right|=0.$$
	For simplicity, let $n_L:=\lambda(L,k)$.
	Then
	$$
	\begin{aligned}
	&\left|S_0^{T_{n_L}}f(\omega,x)-T_{n_L}\alpha_k\right|\\
	\le&\left|S_0^{T_1}f(\omega,x)-T_1\alpha_1\right|+
	\sum_{j=2}^{n_L}\sum_{l=0}^{m(j)-1}\left|S_{t_l^j}^{t_{l+1}^j}f(\sigma^{t_l^j}\omega,\sigma^{t_l^j}x)-(t_{l+1}^j-t_l^j)\alpha'_j\right|\\
	&+T_1\left|\alpha_1-\alpha_k\right|+\sum_{j=2}^{n_L}(T_j-T_{j-1})\left|\alpha'_j-\alpha_k\right|.
	\end{aligned}
	$$
	For each $j\ge 2$ and $0\le l<m(j)$,
	$$
	\begin{aligned}
	&\left|S_{t_l^j}^{t_{l+1}^j}f(\sigma^{t_l^j}\omega,\sigma^{t_l^j}x)-(t_{l+1}^j-t_l^j)\alpha'_j\right|
	\le\sum_{i=1}^{t_{l+1}^j-t_l^j}\mathrm{var}_i f+(t_{l+1}^j+t_l^j)\epsilon'_j.
	\end{aligned}
	$$
	Thus,
	$$
	\begin{aligned}
	&\left|S_0^{T_{n_L}}f(\omega,x)-T_{n_L}\alpha_k\right|\\
	\le&\sum_{i=1}^{T_1}\mathrm{var}_i f+T_1\epsilon'_1
	+\sum_{j=2}^{n_L}\sum_{l=0}^{m(j)-1}\sum_{i=1}^{t_{l+1}^j-t_l^j}\mathrm{var}_i f+\sum_{j=2}^{n_L}\sum_{l=0}^{m(j)-1}(t_{l+1}^j+t_l^j)\epsilon'_j\\
	&+T_1\left|\alpha_1-\alpha_k\right|+\sum_{j=2}^{n_L-1}(T_j-T_{j-1})\left|\alpha'_j-\alpha_k\right|\quad(\text{Since } \alpha'_{n_L}=\alpha_k)\\
	\le&\sum_{i=1}^{T_1}\mathrm{var}_i f+T_1\epsilon'_1
	+\sum_{j=2}^{n_L}\sum_{l=0}^{m(j)-1}\sum_{i=1}^{t_{l+1}^j-t_l^j}\mathrm{var}_i f+\sum_{j=2}^{n_L}\sum_{l=0}^{m(j)-1}(t_{l+1}^j+t_l^j)\epsilon'_j\\
	&+2T_{n_L-1}\sup_{\alpha\in\mathcal{P}_A}|\alpha|.\quad(\text{Since $\mathcal{P}_A$ is bounded})
	\end{aligned}
	$$
	Since $\lim_{i\To\infty}\mathrm{var}_i f=0$ and $\lim_{j\To\infty}t_{l+1}^j-t_l^j=\infty$, we have
	$$\lim_{L\To\infty}\frac1{T_{n_L}}\left(\sum_{i=1}^{T_1}\mathrm{var}_i f+\sum_{j=2}^{n_L}\sum_{l=0}^{m(j)-1}\sum_{i=1}^{t_{l+1}^j-t_l^j}\mathrm{var}_i f\right)=0.$$
	On the other hand,
	$$
	\begin{aligned}
	\sum_{j=2}^{n_L}\sum_{l=0}^{m(j)-1}(t_{l+1}^j+t_l^j)\epsilon'_j
	\le&\sum_{j=2}^{n_L}\sum_{l=0}^{m(j)-1}(r_j+1)r_j^lT_{j-1}\epsilon'_j\\
	\le&\sum_{j=2}^{n_L}\frac{3\epsilon'_jT_{j-1}(r_j^{m(j)+1}-1)}{r_j-1}\\
	\le&\sum_{j=2}^{n_L}\frac{3\epsilon'_jT_j}{r_j-1}.
	\end{aligned}
	$$
	Since $\lim_{j\To\infty}\frac{\sum_{i=1}^{j}T_i}{T_{j+1}}=0$ and $\lim_{j\To\infty}\frac{\epsilon'_j}{r_j-1}=0$, as $L\To\infty$,
	$$\frac1{T_{n_L}}\sum_{j=2}^{n_L}\sum_{l=0}^{m(j)-1}(t_{l+1}^j+t_l^j)\epsilon'_j
	\le\sum_{j=2}^{n_L-1}\frac{3\epsilon'_jT_j}{T_{n_L}(r_j-1)}+\frac{3\epsilon'_{n_L}}{r_{n_L}-1}\To0.
	$$
	Thus, it proves the claim. By the denseness of $C'$, we have $A(\omega,x)\supset C$ for every $x\in M$.
	
	To prove (2), for any $j\ge 2$, $0\le l<m(j)$ and $t_l^j\le n<t_{l+1}^j$, we have
	$$\#W_{T_{j-1}}^{t_l^j}(M)\ge\#W_{T_{j-1}}^{t_l^j}(M'_j)\ge \exp((t_l^j-T_{j-1})(h_{\mu'_j}-h_\nu-\delta)).$$
	In particular, we have
	$$\#W_{T_{j-1}}^{T_j}(M)\ge \exp((T_j-T_{j-1})(h_{\mu'_j}-h_\nu-\delta).$$
	So
	$$
	\begin{aligned}
	\#W_{0}^{t_l^j}(M)\ge& \#W_{T_{j-1}}^{t_l^j}(M)\cdot\prod_{i=1}^{j-1}\#W_{T_{i-1}}^{T_i}(M)\\
	\ge&\exp\left((t_l^j-T_{j-1})(h_{\mu'_j}-h_\nu-\delta)+\sum_{i=1}^{j-1}(T_i-T_{i-1})(h_{\mu'_i}-h_\nu-\delta)\right)\\
	\ge&\exp\left(t_l^j(\inf_{\alpha\in C}\sup\{h_\mu:\mu\in\mathcal{E}_\nu(\Gamma_A)\text{ and }\int f d\mu=\alpha\}-h_\nu-2\delta)\right).
	\end{aligned}
	$$
	Next, we define a probability measure $\hat{\mu}$ as follows.
	For any $W\in\Sigma_n$, there exists unique $j\in\N$ and $0\le l<m(j)$ such that $t_l^j<n\le t_{l+1}^j$. Define
	$$\hat{\mu}([W]):=\frac{\#\{W'\in W_0^{t_{l+1}^j}(M):[W]\supset[W']\}}{\#W_0^{t_{l+1}^j}(M)}.$$
	Indeed, $\hat{\mu}$ is a well defined probability measure on $M$.
	So for any $x\in M$,
	$$
	\begin{aligned}
	&\liminf_{n\To\infty}-\frac1n\log\hat{\mu}([x|_{[0,n)}])\\
	\ge&\liminf_{n\To\infty}\frac{\log\#W_0^{t_l^j}(M)}{n}\\
	\ge&\liminf_{n\To\infty}\frac{t_l^j}{n}\left(\inf_{\alpha\in C}\sup\{h_\mu:\mu\in\mathcal{E}_\nu(\Gamma_A)\text{ and }\int f d\mu=\alpha\}-h_\nu-2\delta\right)\\
	\ge&\liminf_{j\To\infty}\frac{1}{r_j}\left(\inf_{\alpha\in C}\sup\{h_\mu:\mu\in\mathcal{E}_\nu(\Gamma_A)\text{ and }\int f d\mu=\alpha\}-h_\nu-2\delta\right)\\
	=&\inf_{\alpha\in C}\sup\{h_\mu:\mu\in\mathcal{E}_\nu(\Gamma_A)\text{ and }\int f d\mu=\alpha\}-h_\nu-2\delta.
	\end{aligned}
	$$
	Then by Lemma \ref{l:FL}, $h^B(M)\ge\inf_{\alpha\in C}\sup\{h_\mu:\mu\in\mathcal{E}_\nu(\Gamma_A)\text{ and }\int f d\mu=\alpha\}-h_\nu-2\delta$.
	The rest of the proof is similar to the last discussions on the proof of \cite[Proposition 4.3]{BRS20}, which can find a map $\pi:\Sigma\To\Sigma_A$ such that $\widetilde{M}:=\pi(M)$ satisfies (1) and (2).
\end{proof}

\section{Bowen entropy of $\Delta^\omega_{\mathrm{equ}}(C)$}\label{s:5}
In this section, we fix a compact and connected subset $C\subset\mathcal{P}_A^o$, a continuous function $f:\Gamma_A\To\R^d$ with bounded variation and a quasi-Bernoulli $\sigma$-invariant ergodic measure $\nu$ on $\Omega$.

By Section \ref{s:4}, we have
$$h^B(\Delta^\omega_{\mathrm{equ}}(C))\le h^B(\Delta^\omega_{\mathrm{sup}}(C))=\inf_{\alpha\in C}\sup\{h_\mu:\mu\in\mathcal{E}_\nu(\Gamma_A)\text{ and }\int f d\mu=\alpha\}-h_\nu.$$

To estimate the lower bound, we need to adjust the sequence $\{\alpha_k\}_{k\in\N}$ in the construction in Lemma \ref{l:Moren}. We fix an arbitrary decreasing sequence $\{\epsilon_L\}_{L\in\N}$ with $\lim_{L\To\infty}\epsilon_L=0$. Since $C$ is compact and connected, we can find an index set $I=\{(L,k):L\in\N\text{ and } 1\le k \le J(L)\}$ and a sequence $\{\alpha_{L,k}\}_{(L,k)\in I}\subset C$ such that
\begin{itemize}
	\item[(i)] $|\alpha_{L,k}-\alpha_{L,k+1}|<\epsilon_L$ for each $L\in\N$ and $1\le k\le J(L)-1$;
	\item[(ii)] $|\alpha_{L,J(L)}-\alpha_{L+1,1}|<\epsilon_L$ for each $L\in\N$.
	\item[(iii)] $\bigcup_{k=1}^{J(L)}B(\alpha_{L,k},\epsilon_L)\supset C$.
\end{itemize}
Now we define a map $\lambda:I\To\{1,2,\dots\}$ by $\lambda(L,k)=k+\sum_{i=1}^{L-1}J(i)$.
It is easy to see that $\lambda$ is a bijection.
Define $\{\epsilon'_j\}_{j\in\N}$ and $\{\alpha'_j\}_{j\in\N}$ by $\epsilon'_j=\epsilon_L$ and $\alpha'_j=\alpha_{L,k}$ where $j=\lambda(L,k)$. Then we have $|\alpha'_j-\alpha'_{j+1}|<\epsilon'_j$ for any $j\in\N$.
	
\begin{lemma}\label{l:Moren2}
	For $\nu$-a.e. $\omega\in\Omega$,
	$$h^B(\Delta^\omega_{\mathrm{equ}}(C))\ge\inf_{\alpha\in C}\sup\{h_\mu:\mu\in\mathcal{E}_\nu(\Gamma_A)\text{ and }\int f d\mu=\alpha\}-h_\nu.$$
\end{lemma}

\begin{proof}
	Fix any $\delta>0$, we will construct a subset $M$ which is similar to Lemma \ref{l:Moren}.
	For each $j\in\N$, we choose $\mu_j\in\mathcal{E}_\nu(\Gamma_A)$ such that $\int f d\mu_j=\alpha'_j$ and
	$$h_{\mu_j}>\sup\{h_\mu:\mu\in\mathcal{E}_\nu(\Gamma_A)\text{ and }\int f d\mu=\alpha'_j\}-\delta.$$
	By Rohlin's Disintegration Theorem, denote the disintegration of $\mu_j$ by $$\mu_j=\int\mu_j^{\omega,\xi}d\nu(\omega).$$
	By ergodic theorem, there exists $G_j\subset \Gamma_A$ such that $\mu_j(G_j)=1$ and
	$$\lim_{n\To\infty}\frac1nS_0^nf(\omega,x)=\alpha'_j\text{ for each }(\omega,x)\in G_j.$$
	Therefore, let $$G'_j=\{\omega\in\Omega:\mu_j^{\omega,\xi}(\Delta^\omega_{\mathrm{equ}}(\{\alpha'_j\}))=1\},$$
	and then we have $\nu(G'_j)=1$. Let $G^*=\bigcap_{j\in\N}G'_j$. Then we have $\nu(G^*)=1$ and for any $j\in\N$ and $\omega\in G^*$, we have $\mu_j^{\omega,\xi}(\Delta^\omega_{\mathrm{equ}}(\{\alpha'_j\}))=1$.
	On the other hand, there exists $G^{**}\subset\Omega$ such that $\nu(G^{**})=1$ and for any $\omega\in G^{**}$ and $j\in\N$,
	$$\lim_{n\To\infty}-\frac1n\log\mu_j^{\omega,\xi}([x|_{[0,n)}])=h_{\mu_j}-h_\nu\text{ for $\mu_j^{\omega,\xi}$-a.e. $x\in\Sigma_A$}.$$
	
	Let $G=G^*\cap G^{**}$.
	Now we fix $\omega\in G$.
	For each  $j\in\N$, there exists $M_j\subset\Sigma_A$ such that $\mu_j^{\omega,\xi}(M_j)>1-\delta$, and for each $\epsilon'_j$ there exists $N_j$ such that for any $x\in M_j$ and $n>N_j$,
	$$\left|\frac1nS_0^nf(\omega,x)-\alpha'_j\right|<\epsilon'_j,$$
	and
	$$\left|-\frac1n\log\mu_j^{\omega,\xi}([x|_{[0,n)}])-(h_{\mu_j}-h_\nu)\right|<\epsilon'_j.$$
	
	Next, similar to the argument in Lemma \ref{l:Moren}, for each $j\in\N$, we choose a sequence $\{n_i^j\}_{i\ge 1}$ such that for any $i\ge 1$ and $n>n_i^j$, we have
	$$\log\#W_{n_i^j}^n(M_j)>(n-n_i^j)(h_{\mu_j}-h_\nu-\delta).$$

	Now we define a sequence $\{T_j\}_{j\in\N}$ satisfying the following properties:
	\begin{itemize}
		\item $T_0=0$, $T_1>N_1$;
		\item $T_j>N_{j+1}$ and $T_j\in\{n_i^{j+1}\}_{i\in\N}$;
		\item $\lim_{j\To\infty}\frac{\sum_{i=1}^{j}T_i}{T_{j+1}}=0$;
		\item $\log\#W_0^n(M_j)\ge n(h_{\mu_j}-h_\nu-\delta)$ for all $n>T_j$;
		\item $\lim_{j\To\infty}T_j\epsilon'_j=\infty$.
	\end{itemize}
	For $j\ge 2$, define $1<r_j\le 2$ and $m(j)\in\N$ such that
	$$r_j^{m(j)}=\frac{T_j}{T_{j-1}},\,\lim_{j\To\infty}r_j=1,\,\lim_{j\To\infty}\frac{r_j-1}{\epsilon'_j}=\infty.$$
	For $j\ge 2$, let $t_l^j=\lfloor (r_j)^lT_{j-1}\rfloor$ for $l=0,1,\dots,m(j)$. Define a set
	$$
	\begin{aligned}
	M=\{x\in\Sigma:&x|_{[0,T_1)}\in W_0^{T_1}(M_1),\\ &x|_{[t_l^j,t_{l+1}^j)}\in W_{t_l^j}^{t_{l+1}^j}(M_j)\text{ for }j\ge 2\text{ and }0\le l\le m(j)-1\}.
	\end{aligned}
	$$

	Now we will show that the set $M$ satisfies that
	\begin{itemize}
		\item[(1)] for any $x\in M$, $A(\omega,x)=C$;
		\item[(2)] $h^B(M)\ge \inf_{\alpha\in C}\sup\{h_\mu:\mu\in\mathcal{E}_\nu(\Gamma_A)\text{ and }\int f d\mu=\alpha\}-h_\nu-\delta$.
	\end{itemize}

	To prove (1), we will show that for large enough $n$, let $k,l\in\N$ with $t^k_l\le n<t^k_{l+1}$, and we have
	$$\left|\frac1nS_0^nf(\omega,x)-\alpha'_k\right|=o(1).$$
	
	Similar to the proof of Lemma \ref{l:Moren},
	$$
	\begin{aligned}
	&\left|S_0^nf(\omega,x)-n\alpha'_k\right|\\
	\le&\sum_{i=1}^{T_1}\mathrm{var}_i f
	+\sum_{j=2}^{k-1}\sum_{l'=0}^{m(j)-1}\sum_{i=1}^{t_{l'+1}^j-t_{l'}^j}\mathrm{var}_i f
	+\sum_{l'=0}^{l-1}\sum_{i=1}^{t_{l'+1}^k-t_{l'}^k}\mathrm{var}_i f+\sum_{i=t_{l+1}^k-t_l^k}^{t_{l+1}^k-n}\mathrm{var}_i f
	\\&+T_1\epsilon'_1+\sum_{j=2}^{k-1}\sum_{l=0}^{m(j)-1}(t_{l+1}^j+t_l^j)\epsilon'_j
	+\sum_{l'=0}^{l-1}(t_{l'+1}^k+t_{l'}^k)\epsilon'_k+(n+t_l^k)\epsilon'_k
	\\&+T_1\left|\alpha_1-\alpha'_k\right|+\sum_{j=2}^{k-1}(T_j-T_{j-1})\left|\alpha'_j-\alpha'_k\right|
	\\\le&\sum_{i=1}^{T_1}\mathrm{var}_i f
	+\sum_{j=2}^{k-1}\sum_{l'=0}^{m(j)-1}\sum_{i=1}^{t_{l'+1}^j-t_{l'}^j}\mathrm{var}_i f
	+\sum_{l'=0}^{l-1}\sum_{i=1}^{t_{l'+1}^k-t_{l'}^k}\mathrm{var}_i f+\sum_{i=t_{l+1}^k-t_l^k}^{t_{l+1}^k-n}\mathrm{var}_i f
	\\&+T_1\epsilon'_1+\sum_{j=2}^{k-1}\sum_{l=0}^{m(j)-1}(t_{l+1}^j+t_l^j)\epsilon'_j
	+\sum_{l'=0}^{l-1}(t_{l'+1}^k+t_{l'}^k)\epsilon'_k+(n+t_l^k)\epsilon'_k\\
	&+2T_{k-2}\sup_{\alpha\in\mathcal{P}_A}|\alpha|+(T_{k-1}-T_{k-2})\epsilon'_k.
	\end{aligned}
	$$
	As same as the proof in Lemma \ref{l:Moren}, we have $A(\omega,  x)\supset C$. And by the arbitrariness of $n$, we have $A(\omega, x)\subset C$, which means that $A(\omega, x)=C$.
	
	To prove (2), it is similar to the proof of Lemma \ref{l:Moren} that we can construct a subset $\widetilde{M}\subset \Delta^\omega_{\mathrm{equ}}(C)$ with
	$$h^B(\widetilde{M})\ge\inf_{\alpha\in C}\sup\{h_\mu:\mu\in\mathcal{E}_\nu(\Gamma_A)\text{ and }\int f d\mu=\alpha\}-h_\nu-\delta.$$
	By the arbitrariness of $\delta$, it ends the proof.
\end{proof}

Sum up with Section \ref{s:3},Section \ref{s:4} and Section \ref{s:5}, we can prove Theorem \ref{t:1}.

\begin{proof}[Proof of Theorem \ref{t:1}]
	We prove it for each part.
	\begin{itemize}
		\item[(1):] Notice that  $$h^B(\Delta^\omega_{\mathrm{cap}}(C))\ge h^B(\Delta^\omega_{\mathrm{sub}}(C))\ge \sup_{\alpha\in C}h^B(\Delta^\omega_{\mathrm{equ}}(\{\alpha\})).$$
		Then by Theorem \ref{t:fo} and Lemma \ref{l:1upperbound}, it ends the proof.
		
		\item[(2):] By Theorem \ref{t:fo}, Lemma \ref{l:2upperbound} and Lemma \ref{l:Moren}, it ends the proof.
		
		\item[(3):] By (2), Lemma \ref{l:Moren2} and the fact that $$h^B(\Delta^\omega_{\mathrm{equ}}(C))\le h^B(\Delta^\omega_{\mathrm{sup}}(C)),$$ it ends the proof.
	\end{itemize}
\end{proof}

\section{Packing entropy of $\Delta^\omega_{\mathrm{sub}}(C)$ and $\Delta^\omega_{\mathrm{equ}}(C)$}\label{s:6}
In this section, we fix a continuous function $f:\Gamma_A\To\R^d$ with bounded variation and a quasi-Bernoulli $\sigma$-invariant ergodic measure $\nu$ on $\Omega$. To calculate the packing entropy of $\Delta^\omega_{\mathrm{sub}}(C)$ and $\Delta^\omega_{\mathrm{equ}}(C)$, we show the upper bound of $h^P(\Delta^\omega_{\mathrm{sub}}(C))$ for any $C\subset \mathcal{P}^o_A$ and the lower bound of $h^P(\Delta^\omega_{\mathrm{equ}}(C))$ for any compact and connected subset $C\subset \mathcal{P}_A$.

To estimate the upper bound of $h^P(\Delta^\omega_{\mathrm{sub}}(C))$, we need some notations.
For $C\subset\mathcal{P}_A$, $\omega\in\Omega$, $\delta>0$ and $n\in\N$, let
$$G_\omega(C,n,\delta)=\{x\in\Sigma_A:\frac1nS_0^nf(\omega,x)\in B(C,\delta)\},$$
$$G_\omega(C,\delta)=\limsup_{n\To\infty}\frac1n\log \#W_0^n(G_\omega(C,n,\delta))$$
and
$$G_\omega(C)=\lim_{\delta\To0}G_\omega(C,\delta),$$
recall that $W_0^n(M)=\{x|_{[0,n)}\in\Sigma_{A,n}:x\in M\}$. Since $G_\omega(C,\delta)$ is decreasing as $\delta$ decreases, then the limit $G_\omega(C)$ exists and equals to $\inf_{\delta>0}G_\omega(C,\delta)$.

\begin{lemma}\label{l:Packing upperbound}
	For $C\subset \mathcal{P}_A^o$ and $\nu$-a.e. $\omega\in\Omega$, we have
	$$h^P(\Delta^\omega_{\mathrm{sub}}(C))\le\sup_{\alpha\in C}\inf_{p\in\R^d}P_\nu(f_p^\alpha).$$
\end{lemma}

\begin{proof}
	First, we assume that $\mathrm{dist}(C,\R^d\setminus\mathcal{P}_A^o)>0$.
	The proof is divided into two parts:
	\begin{itemize}
		\item[(1)] for any $\omega\in\Omega$, $h^P(\Delta^\omega_{\mathrm{sub}}(C))\le G_\omega(C)$;
		\item[(2)] for $\nu$-a.e. $\omega\in\Omega$, $G_\omega(C)\le\sup_{\alpha\in C}\inf_{p\in\R^d}P(f_p^\alpha,\omega).$
	\end{itemize}
	Part 1: To prove (1), we note that for any $\delta>0$,
	$$\Delta^\omega_{\mathrm{sub}}(C)\subset \bigcup_{N\in\N}\bigcap_{n\ge N}G_\omega(C,n,\delta).$$
	Thus we have
	$$h^P(\Delta^\omega_{\mathrm{sub}}(C))\le\sup_{N\in\N}h^P(\bigcap_{n\ge N}G_\omega(C,n,\delta))\le\sup_{N\in\N}h^{UC}(\bigcap_{n\ge N}G_\omega(C,n,\delta)).$$
	For any $N\in\N$, we have
	$$h^{UC}(\bigcap_{n\ge N}G_\omega(C,n,\delta))\le\limsup_{n\To\infty}\frac1n\log\#W_0^n((G_\omega(C,n,\delta)))=G_\omega(C,\delta).$$
	Thus for any $\delta>0$,
	$$h^P(\Delta^\omega_{\mathrm{sub}}(C))\le G_\omega(C,\delta).$$
	And let $\delta$ tends to $0$, we prove (1).
	
	Part 2: To prove (2), let $C'=\{\alpha_1,\alpha_2,\dots\}\subset C$ be a countable dense subset of $C$.
	Fix any $\eta>0$. Then there exists $k_\eta\in\N$ such that $\bigcup_{k=1}^{k_\eta}B(\alpha_k,\eta)\supset C$.
	Let $s=\sup_{\alpha\in C}\inf_{p\in\R^d}P_\nu(f^\alpha_p)+\eta$.
	For each $1\le k\le k_\eta$, there is $p_k\in\R^d$ with $P_\nu(f^{\alpha_k}_{p_k})<s$. For simplicity, let $f_k:=f^{\alpha_k}_{p_k}$.
	Since $P(f_k,\omega)=P_\nu(f_k)$ for $\nu$-a.e. $\omega\in\Omega$, then there exists $\Omega_0\subset\Omega$ with $\nu(\Omega_0)=1$ such that $P(f_k,\omega)=P_\nu(f_k)$ for any $\omega\in\Omega_0$ and $1\le k\le k_\eta$.
	
	Now fix $\omega\in\Omega_0$. Then for any $n\in\N$ and $\delta>0$,
	$$G_\omega(C,n,\delta)\subset \bigcup_{k=1}^{k_\eta}G_\omega(\{\alpha_k\},n,\delta+\eta),$$
	that is,
	$$\#W_0^n(G_\omega(C,n,\delta))\le \sum_{k=1}^{k_\eta}\#W_0^n(G_\omega(\{\alpha_k\},n,\delta+\eta)).$$
	Then there exists a sequence $\{n_i\}$ and $1\le k(\delta,\eta)\le k_\eta$ such that
	$$G_\omega(C,\delta)=\lim_{i\To\infty}\frac1{n_i}\log\#W_0^{n_i}(G_\omega(C,n_i,\delta))$$
	and for each $i\in\N$,
	$$\#W_0^{n_i}(G_\omega(\{\alpha_{k(\delta,\eta)}\},n_i,\delta+\eta))=\max_{1\le k\le k_\eta}\#W_0^{n_i}(G_\omega(\{\alpha_k\},n_i,\delta+\eta)).$$
	For $\delta<\eta$, we have
	$$
	\begin{aligned}
	G_\omega(C,\delta)
	\le&\lim_{i\To\infty}\frac1{n_i}\log(k_\eta\cdot\#W_0^{n_i}(G_\omega(\{\alpha_{k(\delta,\eta)}\},n_i,\delta+\eta)))\\
	\le&\lim_{i\To\infty}\frac1{n_i}\log\#W_0^{n_i}(G_\omega(\{\alpha_{k(\delta,\eta)}\},n_i,2\eta))\\
	\le&G_\omega(\{\alpha_{k(\delta,\eta)}\},2\eta).
	\end{aligned}
	$$
	Therefore, there exists $1\le k^*\le k_\eta$ such that
	$G_\omega(C)\le G_\omega(\{\alpha_{k^*}\},2\eta)$.
	
	Since $s>P_\nu(f_{k^*})$, then by Lemma \ref{l:p's_bound}, we have
	$$|p_{k^*}|<\frac{2P_\nu(f_{k^*})}{\mathrm{dist}(C,\R^d\setminus\mathcal{P}_A^o)}<\frac{2s}{\mathrm{dist}(C,\R^d\setminus\mathcal{P}_A^o)}.$$
	Thus
	$$
	\begin{aligned}
	s>&P_\nu(f_{k^*})=P(f_{k^*},\omega)\\
	=&\limsup_{n\To\infty}\frac1n\log\sum_{W\in\Sigma_{A,n}}\sup_{x\in [W]}e^{S_0^nf_{k^*}(\omega,x)}\\
	\ge&\limsup_{n\To\infty}\frac1n\log\sum_{W\in W_0^n(G_\omega(\{\alpha_{k^*}\},n,2\eta))}\sup_{x\in [W]}e^{S_0^nf_{k^*}(\omega,x)}
	\end{aligned}
	$$
	For each $W\in W_0^n(G_\omega(\{\alpha_{k^*}\},n,2\eta))$, we choose some $x_W\in G_\omega(\{\alpha_{k^*}\},n,2\eta)$ with $(x_W)|_{[0,n)}=W$.
	Then
	$$
	\begin{aligned}
	s>&\limsup_{n\To\infty}\frac1n\log\sum_{W\in W_0^n(G_\omega(\{\alpha_{k^*}\},n,2\eta))}\sup_{x\in [W]}e^{S_0^nf_{k^*}(\omega,x)}\\
	\ge&\limsup_{n\To\infty}\frac1n\log\sum_{W\in W_0^n(G_\omega(\{\alpha_{k^*}\},n,2\eta))}e^{\langle p_{k^*},S_0^nf(\omega,x_W)-n\alpha_{k^*}\rangle }\\
	\ge&\limsup_{n\To\infty}\frac1n\log\sum_{W\in W_0^n(G_\omega(\{\alpha_{k^*}\},n,2\eta))}e^{-|p_{k^*}|\cdot 2n\eta}\\
	\ge&G_\omega(\{\alpha_{k^*}\},2\eta)-\frac{4s\eta}{\mathrm{dist}(C,\R^d\setminus\mathcal{P}_A^o)}\\
	\ge&G_\omega(C)-\frac{4s\eta}{\mathrm{dist}(C,\R^d\setminus\mathcal{P}_A^o)}.
	\end{aligned}
	$$
	Let $\eta\To0$, we prove (2). Although the set $\Omega_0$ may depend on $\eta$, we can choose $\Omega'_0$ does not depend on $\eta$ by choosing $\eta\in\mathbb{Q}$ and the intersection of the corresponding $\Omega_0$.
	
	Now, for general $C\subset \mathcal{P}_A^o$, let $C_n=\{
	\alpha\in C:\mathrm{dist}(\alpha,\R^d\setminus\mathcal{P}_A^o)>\frac1n\}$. Then we note that $C=\bigcup_{n\ge1}C_n$. Since $A(\omega,x)$ is compact,
	$$\Delta^\omega_{\mathrm{sub}}(C)=\bigcup_{n\ge1}\Delta^\omega_{\mathrm{sub}}(C_n).$$
	Thus
	$$h^P(\Delta^\omega_{\mathrm{sub}}(C))\le\sup_{n\ge 1}\sup_{\alpha\in C_n}\inf_{p\in\R^d}P_\nu(f^\alpha_p)=\sup_{\alpha\in C}\inf_{p\in\R^d}P_\nu(f^\alpha_p).$$
\end{proof}

Next, to estimate the lower bound of $h^P(\Delta^\omega_{\mathrm{equ}}(C))$, we fix some compact and connected subset $C\subset\mathcal{P}_A$.
Similar to estimate the lower bound of $h^B(\Delta^\omega_{\mathrm{equ}}(C))$, We fix an arbitrary decreasing sequence $\{\epsilon_L\}_{L\in\N}$ with $\lim_{L\To\infty}\epsilon_L=0$. Since $C$ is compact and connected, we can find an index set $I=\{(L,k):L\in\N\text{ and } 1\le k \le J(L)\}$ and a sequence $\{\alpha_{L,k}\}_{(L,k)\in I}\subset C$ such that
\begin{itemize}
	\item[(i)] $|\alpha_{L,k}-\alpha_{L,k+1}|<\epsilon_L$ for each $L\in\N$ and $1\le k\le J(L)-1$;
	\item[(ii)] $|\alpha_{L,J(L)}-\alpha_{L+1,1}|<\epsilon_L$ for each $L\in\N$;
	\item[(iii)] $\bigcup_{k=1}^{J(L)}B(\alpha_{L,k},\epsilon_L)\supset C$;
	\item[(iv)] $\sup\{h_\mu:\mu\in\mathcal{E}_\nu(\Gamma_A),\int f d\mu=\alpha_{L,J(L)}\}\ge\sup\{h_\mu:\mu\in\mathcal{E}_\nu(\Gamma_A),\int f d\mu\in C\}-\epsilon_L$.
\end{itemize}
Now we define a map $\lambda:I\To\{1,2,\dots\}$ by $\lambda(L,k)=k+\sum_{i=1}^{L-1}J(i)$.
It is easy to see that $\lambda$ is a bijection.
Define $\{\epsilon'_j\}_{j\in\N}$ and $\{\alpha'_j\}_{j\in\N}$ by $\epsilon'_j=\epsilon_L$ and $\alpha'_j=\alpha_{L,k}$ where $j=\lambda(L,k)$. Then we have $|\alpha'_j-\alpha'_{j+1}|<\epsilon'_j$ for any $j\in\N$.

\begin{lemma}\label{l:Packing lowerbound}
	If $C\subset\mathcal{P}_A$ is compact and connected, then for $\nu$-a.e. $\omega\in\Omega$,
	$$h^P(\Delta^\omega_{\mathrm{equ}}(C))\ge\sup\{h_\mu:\mu\in\mathcal{E}_\nu(\Gamma_A),\int f d\mu\in C\}-h_\nu.$$
\end{lemma}

\begin{proof}
	Fix any $\delta>0$, we will construct a subset $M$ which is similar to Lemma \ref{l:Moren2}.
	For each $j\in\N$, we choose $\mu_j\in\mathcal{E}_\nu(\Gamma_A)$ such that $\int f d\mu_j=\alpha'_j$ and
	$$h_{\mu_j}>\sup\{h_\mu:\mu\in\mathcal{E}_\nu(\Gamma_A)\text{ and }\int f d\mu=\alpha'_j\}-\delta.$$
	By Rohlin's Disintegration Theorem, denote the disintegration of $\mu_j$ by $$\mu_j=\int\mu_j^{\omega,\xi}d\nu(\omega).$$
	By ergodic theorem, there exists $G_j\subset \Gamma_A$ such that $\mu_j(G_j)=1$ and
	$$\lim_{n\To\infty}\frac1nS_0^nf(\omega,x)=\alpha'_j\text{ for each }(\omega,x)\in G_j.$$
	Therefore, let $$G'_j=\{\omega\in\Omega:\mu_j^{\omega,\xi}(\Delta^\omega_{\mathrm{equ}}(\{\alpha'_j\}))=1\},$$
	and then we have $\nu(G'_j)=1$. Let $G^*=\bigcap_{j\in\N}G'_j$. Then we have $\nu(G^*)=1$ and for any $j\in\N$ and $\omega\in G^*$, we have $\mu_j^{\omega,\xi}(\Delta^\omega_{\mathrm{equ}}(\{\alpha'_j\}))=1$.
	On the other hand, there exists $G^{**}\subset\Omega$ such that $\nu(G^{**})=1$ and for any $\omega\in G^{**}$ and $j\in\N$,
	$$\lim_{n\To\infty}-\frac1n\log\mu_j^{\omega,\xi}([x|_{[0,n)}])=h_{\mu_j}-h_\nu\text{ for $\mu_j^{\omega,\xi}$-a.e. $x\in\Sigma_A$}.$$
	
	Let $G=G^*\cap G^{**}$.
	Now we fix $\omega\in G$.
	For each  $j\in\N$, there exists $M_j\subset\Sigma_A$ such that $\mu_j^{\omega,\xi}(M_j)>1-\delta$, and for each $\epsilon'_j$ there exists $N_j$ such that for any $x\in M_j$ and $n>N_j$,
	$$\left|\frac1nS_0^nf(\omega,x)-\alpha'_j\right|<\epsilon'_j,$$
	and
	$$\left|-\frac1n\log\mu_j^{\omega,\xi}([x|_{[0,n)}])-(h_{\mu_j}-h_\nu)\right|<\epsilon'_j.$$
	
	Next, similar to the argument in Lemma \ref{l:Moren2}, for each $j\in\N$, we choose a sequence $\{n_i^j\}_{i\ge 1}$ such that for any $i\ge 1$ and $n>n_i^j$, we have
	$$\log\#W_{n_i^j}^n(M_j)>(n-n_i^j)(h_{\mu_j}-h_\nu-\delta).$$
	
	Recall that there exists $r\in\N$ such that every entry of $A^r$ is strictly positive.
	Now we define a sequence $\{T_j\}_{j\in\N}$ satisfying the following properties:
	\begin{itemize}
		\item $T_0=0$, $T_1>N_1+r$;
		\item $T_j>N_{j+1}+r$ and $T_j\in\{n_i^{j+1}\}_{i\in\N}$;
		\item $\lim_{j\To\infty}\frac{\sum_{i=1}^{j}T_i}{T_{j+1}}=0$;
		\item $\log\#W_0^n(M_j)\ge n(h_{\mu_j}-h_\nu-\delta)$ for all $n>T_j$;
		\item $\lim_{j\To\infty}T_j\epsilon'_j=\infty$.
	\end{itemize}
	For $j\ge 2$, define $1<r_j\le 2$ and $m(j)\in\N$ such that
	$$r_j^{m(j)}=\frac{T_j}{T_{j-1}},\,\lim_{j\To\infty}r_j=1,\,\lim_{j\To\infty}\frac{r_j-1}{\epsilon'_j}=\infty.$$
	For $j\ge 2$, let $t_l^j=\lfloor (r_j)^lT_{j-1}\rfloor$ for $l=0,1,\dots,m(j)$. Define a set
	$$
	\begin{aligned}
	M=\{x\in\Sigma_A:&x|_{[0,T_1-r)}\in W_0^{T_1-r}(M_1),\\ &x|_{[t_l^j,t_{l+1}^j-r)}\in W_{t_l^j}^{t_{l+1}^j-r}(M_j)\text{ for }j\ge 2\text{ and }0\le l\le m(j)-1\}.
	\end{aligned}
	$$
	We note that for any choice $\{W_{j,l}\in W_{t_l^j}^{t_{l+1}^j}(M_j):j\ge 2,\,0\le l\le m(j)-1\}$, there is $x\in M$ such that $x|_{[t_l^j,t_{l+1}^j-r)}=W_{j,l}$ for any $j\ge 2,\,0\le l\le m(j)-1$.
	
	Now we will show that the set $M$ satisfies that
	\begin{itemize}
		\item[(1)] for any $x\in M$, $A(\omega,x)=C$;
		\item[(2)] $h^P(M)\ge\sup\{h_\mu:\mu\in\mathcal{E}_\nu(\Gamma_A)\text{ and }\int f d\mu\in C\}-h_\nu-2\delta$.
	\end{itemize}

	To prove (1), we will show that for large enough $n$, let $k,l\in\N$ with $t^k_l\le n<t^k_{l+1}$, and we have
	$$\left|\frac1nS_0^nf(\omega,x)-\alpha'_k\right|=o(1).$$
	
	Similar to the proof of Lemma \ref{l:Moren2},
	for each $j\ge 2$ and $0\le l<m(j)$,
	$$
	\left|S_{t_l^j}^{t_{l+1}^j}f(\sigma^{t_l^j}\omega,\sigma^{t_l^j}x)-(t_{l+1}^j-t_l^j)\alpha'_j\right|\\
	\le2r|f|+\sum_{i=1}^{t_{l+1}^j-t_l^j-r}\mathrm{var}_i f+(t_{l+1}^j+t_l^j)\epsilon'_j,
	$$
	where $|f|=\sup_{(\omega,x)\in\Gamma_A}|f(\omega,x)|$.
	Then
	$$
	\begin{aligned}
	&\left|S_0^nf(\omega,x)-n\alpha'_k\right|
	\\\le&\sum_{i=1}^{T_1-r}\mathrm{var}_i f
	+\sum_{j=2}^{k-1}\sum_{l'=0}^{m(j)-1}\sum_{i=1}^{t_{l'+1}^j-t_{l'}^j-r}\mathrm{var}_i f
	+\sum_{l'=0}^{l-1}\sum_{i=1}^{t_{l'+1}^k-t_{l'}^k-r}\mathrm{var}_i f+\sum_{i=t_{l+1}^k-t_l^k}^{t_{l+1}^k-n}\mathrm{var}_i f
	\\&+T_1\epsilon'_1+\sum_{j=2}^{k-1}\sum_{l=0}^{m(j)-1}(t_{l+1}^j+t_l^j)\epsilon'_j
	+\sum_{l'=0}^{l-1}(t_{l'+1}^k+t_{l'}^k)\epsilon'_k+(n+t_l^k)\epsilon'_k\\
	&+2r|f|\left(2+\sum_{j=2}^{k-1}m(j)+l\right)+2T_{k-2}\sup_{\alpha\in\mathcal{P}_A}|\alpha|+(T_{k-1}-T_{k-2})\epsilon'_k.
	\end{aligned}
	$$
	It is obvious that $\frac{1}{t^k_l}\left(2+\sum_{j=2}^{k-1}m(j)+l\right)=o(1)$.
	So we prove (1).
	
	To prove (2), for any $j\ge 2$, we have
	$$
	\#W_{T_{j-1}}^{T_j}(M)
	\ge\#W_{T_{j-1}}^{T_j-r}(M_j)
	\ge \exp((T_j-r-T_{j-1})(h_{\mu_j}-h_\nu-\delta)).
	$$
	For any large enough $L$ with $\epsilon_L<\delta$, let $n_L=\lambda(L,J(L))$. Then, we have $\alpha'_{n_L}=\alpha_{L,J(L)}$ and
	$$
	\begin{aligned}
	\#W_{0}^{T_{n_L}}(M)\ge& \#W_{T_{n_L-1}}^{T_{n_L}-r}(M_{n_L})\cdot\prod_{i=1}^{n_L-1}\#W_{T_{i-1}}^{T_i-r}(M_i)\\
	\ge&\exp\left((T_{n_L}-r-T_{n_L-1})(h_{\mu_{n_L}}-h_\nu-\delta)+\sum_{i=1}^{n_L-1}(T_i-r-T_{i-1})(h_{\mu_i}-h_\nu-\delta)\right)\\
	\ge&\exp\left((T_{n_L}-r-T_{n_L-1})(\sup\{h_\mu:\mu\in\mathcal{E}_\nu(\Gamma_A),\int fd\mu\in C\}-h_\nu-\epsilon_L-\delta)\right).
	\end{aligned}
	$$
	Next, we define a probability measure $\hat{\mu}$ as follows.
	For any $W\in\Sigma_n$, there exists unique $j\in\N$ and $0\le l<m(j)$ such that $t_l^j<n\le t_{l+1}^j$. Define
	$$\hat{\mu}([W]):=\frac{\#\{W'\in W_0^{t_{l+1}^j}(M):[W]\supset[W']\}}{\#W_0^{t_{l+1}^j}(M)}.$$
	Indeed, $\hat{\mu}$ is a well defined probability measure on $M$.
	So for any $x\in M$,
	$$
	\begin{aligned}
	&\liminf_{L\To\infty}-\frac1{T_{n_L}}\log\hat{\mu}([x|_{[0,T_{n_L})}])\\
	\ge&\liminf_{L\To\infty}\frac{\log\#W_0^{T_{n_L}}(M)}{T_{n_L}}\\
	\ge&\liminf_{L\To\infty}\frac{T_{n_L}-r-T_{n_L-1}}{T_{n_L}}\left(\sup\{h_\mu:\mu\in\mathcal{E}_\nu(\Gamma_A),\int fd\mu\in C\}-h_\nu-\epsilon_L-\delta\right)\\
	=&\sup\{h_\mu:\mu\in\mathcal{E}_\nu(\Gamma_A),\int fd\mu\in C\}-h_\nu-\delta.
	\end{aligned}
	$$
	Then by Lemma \ref{l:DP}, it proves (2).
	
	So we have $h^P(\Delta^\omega_{\mathrm{equ}}(C))\ge h^P(M)\ge\sup\{h_\mu:\mu\in\mathcal{E}_\nu(\Gamma_A),\int f d\mu\in C\}-h_\nu-\delta$. And by the arbitrariness of $\delta$, it ends the proof.
\end{proof}

\begin{remark}
	In Lemma \ref{l:Packing lowerbound}, we note that $C$ is not necessarily the subset of $\mathcal{P}_A^o$. Thus for the sets $\Delta^\omega_{\mathrm{cap}}(C)$ and $\Delta^\omega_{\mathrm{sup}}(C)$, we have
	$$h^P(\Delta^\omega_{\mathrm{cap}}(C))=h^P(\Delta^\omega_{\mathrm{sup}}(C))=\sup\{h_\mu:\mu\in\mathcal{E}_\nu(\Gamma_A)\}-h_\nu=h(\Sigma_A),$$
	since
	$$h^P(\Delta^\omega_{\mathrm{cap}}(C))\ge h^P(\Delta^\omega_{\mathrm{sup}}(C))\ge h^P(\Delta^\omega_{\mathrm{equ}}(\mathcal{P}_A)).$$
\end{remark}

\begin{proof}[Proof of Theorem \ref{t:4}]
	We prove it for each part.
	\begin{itemize}
		\item[(1):] For any $\alpha\in C\subset \mathcal{P}_A^o$, by Lemma \ref{l:Packing lowerbound}, we have
		$$h^P(\Delta^\omega_{\mathrm{equ}}(\{\alpha\}))\ge\sup\{h_\mu:\mu\in\mathcal{E}_\nu(\Gamma_A),\int fd\mu=\alpha\}-h_\nu-\delta.$$
		Since $h^P(\Delta^\omega_{\mathrm{sub}}(C))\ge\sup_{\alpha\in C} h^P(\Delta^\omega_{\mathrm{equ}}(\{\alpha\}))$,
		then by Theorem \ref{t:fo} and Lemma \ref{l:Packing upperbound}, it ends the proof.
		
		\item[(2):] It directly follows from Lemma \ref{l:Packing upperbound}, Lemma \ref{l:Packing lowerbound} and $h^P(\Delta^\omega_{\mathrm{equ}}(C))\le h^P(\Delta^\omega_{\mathrm{sub}}(C))$.
		
	\end{itemize}
\end{proof}

\section{Proof of Theorem \ref{t:ml}}\label{s:7}

In this section, we will prove Theorem \ref{t:ml} by 3 steps. First, we consider that $\Omega=\Sigma=\{1,\dots,K\}$ where $2\le K\in\N$ and $f:=\rho$ be the metric on $\Sigma$.
Then $f$ can be seen as a function on $\Gamma_A$.
By Theorem \ref{t:1} and Theorem \ref{t:4}, we will show that Theorem \ref{t:ml} holds for almost everywhere $\omega$ in an aperiodic and irreducible shift of finite type $\Sigma_A$, that is, Theorem \ref{t:app}.
Second, we construct a map between $\mathrm{ML}_\omega(\Sigma_A,\sigma)$ for different $\omega\in\Sigma_A$.
This map can transfer almost everywhere to everywhere, which can prove that Theorem \ref{t:ml} holds for each $\omega$ in an aperiodic and irreducible shift of finite type $\Sigma_A$, that is, Theorem \ref{t:app1.5}.
Finally, we prove Theorem \ref{t:ml} by a spectral decomposition theorem.

For a $\sigma$-invariant ergodic quasi-Bernoulli measure $\nu$ on $\Sigma$ with $\nu(\Sigma_A)=1$, we have
$$\mathcal{P}_A=\left\{\int \rho d\mu:\mu\in\mathcal{M}_\nu(\Gamma_A)\right\}=[0,\max\mathcal{P}_A].$$
Note that
$$\mathrm{ML}_\omega(\Sigma_A,\sigma)=\bigcup_{\delta>0}\Delta^\omega_{sup}([0,\delta])=\bigcup_{\delta>0}\Delta^\omega_{sup}((0,\delta)).$$
Then for $\nu$-a.e. $\omega\in\Sigma$,
$$
\begin{aligned}
h^B(\mathrm{ML}_\omega(\Sigma_A,\sigma))
=&h^B\left(\bigcup_{\delta>0}\Delta^\omega_{sup}((0,\delta))\right)\\
=&\sup_{\delta\in\mathbb{Q}}\inf_{\alpha\in(0,\delta)}\inf_{p\in\R}P_{\nu}(p(\rho-\alpha))\\
=&\liminf_{\alpha\To0}\inf_{p\in\R}P_{\nu}(p(\rho-\alpha))\\
=&\liminf_{\alpha\To0}h^B(\Delta^\omega_{equ}(\{\alpha\})).
\end{aligned}
$$
By \cite[Theorem 2.4]{BRS20}, it is shown that the map $\alpha\mapsto h^B(\Delta^\omega_{equ}(\{\alpha\}))$ is continuous.
So we have
$$h^B(\mathrm{ML}_\omega(\Sigma_A,\sigma))=h^B(\Delta^\omega_{equ}(\{0\})).$$
Let $g(\alpha)=\inf_{p\in\R}P_{\nu}(p(\rho-\alpha))=\inf_{p\in\R}(P_{\nu}(p\rho)-p\alpha)$ for $\alpha\in\mathcal{P}_A$.

\begin{theorem}\label{t:app}
Assume that $\Sigma_A$ is an aperiodic and irreducible shift of finite type. If $\nu$ is a $\sigma$-invariant ergodic quasi-Bernoulli measure on $\Sigma$ with $\nu(\Sigma_A)=1$, then for $\nu$-a.e. $\omega\in\Sigma_A$,
$$h^B(\mathrm{ML}_\omega(\Sigma_A,\sigma))=0\text{ and }h^P(\mathrm{ML}_\omega(\Sigma_A,\sigma))=h(\Sigma_A),$$
that is,
$$\dim_H(\mathrm{ML}_\omega(\Sigma_A,\sigma))=0\text{ and }\dim_P(\mathrm{ML}_\omega(\Sigma_A,\sigma))=\dim_H(\Sigma_A).$$
\end{theorem}

\begin{proof}
	By \cite[Lemma 5.5]{BRS20}, we have $h^B(\Delta^\omega_{equ}(\{\alpha\}))\le g(\alpha)$ for any $\alpha\in\mathcal{P}_A$ and $\nu$-a.e. $\omega\in\Sigma$.
	So to prove the first formula,  we only need to prove $g(0)\le 0$.
	
	First, fix any $p<0$. For any $\omega\in\Sigma$, $n\in\N$ and $0\le i\le n$, let
	$$A_i=\{W\in\Sigma_n:\#\{0\le j<n:W_j\neq\omega_j\}=i\}.$$
	Thus $\#A_i\le\binom{n}{i}(K-1)^i$.
	For $W\in A_i$ and $x\in[W]$, we have $S_0^n\rho(\omega,x)\ge i.$
	So for $W\in A_i$,
	$$\sup_{x\in[W]}\exp(pS_0^n\rho(\omega,x))\le e^{pi}.$$
	Then
	$$
	\begin{aligned}
	\log\sum_{W\in\Sigma_{A,n}}\sup_{x\in[W]}\exp(pS_0^n\rho(\omega,x))
	\le&\log\sum_{W\in\Sigma_n}\sup_{x\in[W]}\exp(pS_0^n\rho(\omega,x))\\
	\le&\log\sum_{i=0}^{n}\sum_{W\in A_i}\sup_{x\in[W]}\exp(pS_0^n\rho(\omega,x))\\
	\le&\log\sum_{i=0}^{n}\binom{n}{i}(K-1)^ie^{pi}\\
	=&n\log((K-1)e^p+1).\\
	\end{aligned}
	$$
	Thus, for any $p<0$,
	$$P_\nu(p\rho)=\int\left(\limsup_{n\To\infty}\frac1n\log\sum_{W\in\Sigma_{A,n}}\sup_{x\in[W]}\exp(pS_0^n\rho(\omega,x))\right)d\nu(\omega)\le\log((K-1)e^p+1),$$
	which implies that $$g(0)=\inf_{p\in\R}P_\nu(p\rho)\le 0.$$

	For the second formula, by Lemma \ref{l:Packing lowerbound}, we have
	$$h^P\left(\bigcup_{\delta>0}\Delta^\omega_{sup}([0,\delta])\right)\ge h^P\left(\Delta^\omega_{equ}(\mathcal{P}_A)\right)\ge \sup\{h_\mu:\mu\in\mathcal{E}_\nu(\Gamma_A)\}-h_\nu=h(\Sigma_A).$$
\end{proof}

Next, we will show that Theorem \ref{t:app} holds for each $\omega\in\Sigma_A$.

\begin{theorem}\label{t:app1.5}
	Assume that $\Sigma_A$ is an aperiodic and irreducible shift of finite type. Then for any $\omega\in\Sigma_A$,
	$$h^B(\mathrm{ML}_\omega(\Sigma_A,\sigma))=0\text{ and }h^P(\mathrm{ML}_\omega(\Sigma_A,\sigma))=h(\Sigma_A),$$
	that is,
	$$\dim_H(\mathrm{ML}_\omega(\Sigma_A,\sigma))=0\text{ and }\dim_P(\mathrm{ML}_\omega(\Sigma_A,\sigma))=\dim_H(\Sigma_A).$$
\end{theorem}

\begin{proof}
We prove it in two cases: $\Sigma$ and $\Sigma_A$, where the first case aims to show the main idea of proof.
	
\begin{itemize}
	\item[(1)] Simple case: $\Sigma_A=\Sigma$.
\end{itemize}

For $\omega,\omega'\in\Sigma$, define $\varphi_{\omega,\omega'}:\Sigma\To\Sigma$ by
$$\varphi_{\omega,\omega'}(x)_i=\left\{
\begin{aligned}
\omega'_i,\quad&x_i=\omega_i,\\
x_i,\quad&x_i\neq\omega_i\text{ and }x_i\neq\omega'_i,\\
\omega_i,\quad&x_i\neq\omega_i\text{ and }x_i=\omega'_i,\\
\end{aligned}
\right.$$
for $x=x_0x_1\cdots\in\Sigma$.
It is obvious that $\varphi_{\omega,\omega'}$ is continuous and bi-Lipschitz.
And we have the following lemma.
\begin{lemma}\label{l:phi}
	For any $\omega,\omega'\in\Sigma$, we have
	\begin{itemize}
		\item[(i)] $\varphi_{\omega',\omega}(\varphi_{\omega,\omega'}(x))=x$;
		\item[(ii)] $\varphi_{\omega,\omega'}(\Delta_{sup}^\omega([0,\delta]))=\Delta_{sup}^{\omega'}([0,\delta])$.
	\end{itemize}
\end{lemma}

\begin{proof}
	By the definition of $\varphi_{\omega,\omega'}$, it is no hard to prove (i).
	To prove (ii), we note that
	$$\{i\in\N:\omega_i=x_i\}=\{i\in\N:\omega'_i=\varphi_{\omega,\omega'}(x)_i\},$$
	which implies that for $i\in\N$, $$\rho(\sigma^i\omega,\sigma^ix)=\rho(\sigma^i\omega',\sigma^i\varphi_{\omega,\omega'}(x)).$$
	Then it ends the proof of (ii).
\end{proof}

	Since $\varphi_{\omega,\omega'}$ is continuous and bi-Lipschitz, by Lemma \ref{l:phi}, then for $\omega,\omega'\in\Sigma$,
	$$h^B(\mathrm{ML}_\omega(\Sigma,\sigma))=h^B(\mathrm{ML}_{\omega'}(\Sigma,\sigma))\text{ and }h^P(\mathrm{ML}_\omega(\Sigma,\sigma))=h^P(\mathrm{ML}_{\omega'}(\Sigma,\sigma)).$$
	Since there is an ergodic quasi-Bernoulli measure on $\Sigma$, then by Theorem \ref{t:app}, it ends the proof of simple case.

\begin{itemize}
	\item[(2)] General case: aperiodic and irreducible shift of finite type $\Sigma_A$.
\end{itemize}
For convenience, we give some notations.
Since $\Sigma_A$ is aperiodic and irreducible, we choose $r\ge 0$ such that $A^{r+1}>0$.
Then for any two symbols $a,b\in\mathcal{A}$, we fix a connected word $W_{a,b}\in\Sigma_{A,r}$ such that $aW_{a,b}b\in\Sigma_{A,r+2}$.

Fix any $\omega,\omega'\in\Sigma_A$ and $M\in\mathbb{N}$.
We will define a map $\varphi^M_{\omega,\omega'}:\Sigma_A\rightarrow\Sigma_A$ as follows.
For each $k\in\mathbb{N}$, let
$$I^0_k=[a^0_k,a^0_{k+1})\cap\N:=[kM,(k+1)M)\cap\mathbb{N},$$ $$I^1_k=[a^1_k,a^2_k)\cap\N:=[k(M+r),k(M+r)+M)\cap\mathbb{N},$$
and
$$I^2_k=[a^2_k,a^1_{k+1})\cap\N:=[k(M+r)+M,(k+1)(M+r))\cap\mathbb{N}.$$
For any $x\in\Sigma_A$, we choose $x'\in\Sigma_A$ by the following steps:
\begin{itemize}
	\item[(i)] We define $x'$ on $I^1_k$ for each $k\in\N$:
	$$x'|_{I^1_k}=\left\{
	\begin{aligned}
	\omega'|_{I^1_k},\quad&x|_{I^0_k}=\omega|_{I^0_k},\\
	\omega|_{I^0_k},\quad&x|_{I^0_k}=\omega'|_{I^1_k}\neq \omega|_{I^0_k},\\
	x|_{I^0_k},\quad&x|_{I^0_k}\neq \omega|_{I^0_k}\text{ and }x|_{I^0_k}\neq \omega'|_{I^1_k}.\\
	\end{aligned}
	\right.$$

\begin{figure}
	\begin{center}
		\begin{tikzpicture}
		\foreach \x in {2}
		{
			\draw[thin,-](-5,\x)--(-4,\x);
			\draw[thin,-](-1,\x)--(5,\x);
			\draw[very thick,-](-4,\x)--(-1,\x);
			\draw(6,\x)node[left]{$x$};
			\draw(-2.5,\x)node[below]{$(1)$};
			\draw[thick,-](-4,\x-0.1)--(-4,\x+0.1)node[above]{$a^0_{k}$};
			\draw[thick,-](-1,\x-0.1)--(-1,\x+0.1)node[above]{$a^0_{k+1}$};
		}
		\foreach \x in {1}
		{
			\draw[thin,-](-5,\x)--(-4,\x);
			\draw[thin,-](-1,\x)--(5,\x);
			\draw[very thick,dash pattern=on 5pt off 3pt](-4,\x)--(-1,\x);
			\draw(6,\x)node[left]{$\omega$};
			\draw(-2.5,\x)node[below]{$(2)$};
			\draw[thick,-](-4,\x-0.1)--(-4,\x+0.1)node[above]{$a^0_{k}$};
			\draw[thick,-](-1,\x-0.1)--(-1,\x+0.1)node[above]{$a^0_{k+1}$};
		}
		
		\foreach \x in {0}
		{
			\draw[thin,-](-5,\x)--(1,\x);
			\draw[thin,-](4,\x)--(5,\x);
			\draw[very thick,dashdotted](1,\x)--(4,\x);
			\draw(6,\x)node[left]{$\omega'$};
			\draw(2.5,\x)node[below]{$(3)$};
			\draw[thick,-](1,\x-0.1)--(1,\x+0.1)node[above]{$a^1_{k}$};
			\draw[thick,-](4,\x-0.1)--(4,\x+0.1)node[above]{$a^2_{k}$};
		}
		\foreach \x in {-1}
		{
			\draw[thin,-](0,\x)--(1,\x);
			\draw[thin,-](4,\x)--(5,\x);
			\draw(6,\x)node[left]{$x'$};
			\draw(0,\x)node[left]{If $(1)=(2)$:};
			\draw[thick,-](1,\x-0.1)--(1,\x+0.1);
			\draw[thick,-](4,\x-0.1)--(4,\x+0.1);
			\draw[very thick,dashdotted](1,\x)--(4,\x);
			\draw(2.5,\x)node[below]{$=(3)$};
			\draw[thick,-](1,\x-0.1)--(1,\x+0.1)node[above]{$a^1_k$};
			\draw[thick,-](4,\x-0.1)--(4,\x+0.1)node[above]{$a^2_{k}$};
		}
		\foreach \x in {-2}
		{
			\draw[thin,-](0,\x)--(1,\x);
			\draw[thin,-](4,\x)--(5,\x);
			\draw(6,\x)node[left]{$x'$};
			\draw(0,\x)node[left]{If $(1)=(3)\neq(2)$:};
			\draw[thick,-](1,\x-0.1)--(1,\x+0.1);
			\draw[thick,-](4,\x-0.1)--(4,\x+0.1);
			\draw[thick,-](1,\x-0.1)--(1,\x+0.1)node[above]{$a^1_{k}$};
			\draw[thick,-](4,\x-0.1)--(4,\x+0.1)node[above]{$a^2_{k}$};
			\draw[very thick,dash pattern=on 5pt off 3pt](1,\x)--(4,\x);
			\draw(2.5,\x)node[below]{$=(2)$};
		}
		
		\foreach \x in {-3}
		{
			\draw[thin,-](0,\x)--(1,\x);
			\draw[thin,-](4,\x)--(5,\x);
			\draw(6,\x)node[left]{$x'$};
			\draw(0,\x)node[left]{Otherwise:};
			\draw[thick,-](1,\x-0.1)--(1,\x+0.1);
			\draw[thick,-](4,\x-0.1)--(4,\x+0.1);
			\draw[thick,-](1,\x-0.1)--(1,\x+0.1)node[above]{$a^1_{k}$};
			\draw[thick,-](4,\x-0.1)--(4,\x+0.1)node[above]{$a^2_{k}$};
			\draw[very thick,-](1,\x)--(4,\x);
			\draw(2.5,\x)node[below]{$=(1)$};
		}
		\end{tikzpicture}
	\end{center}
	\caption{The symbols of $x'$ on $I^1_k$}
	\label{f:phi1}
\end{figure}
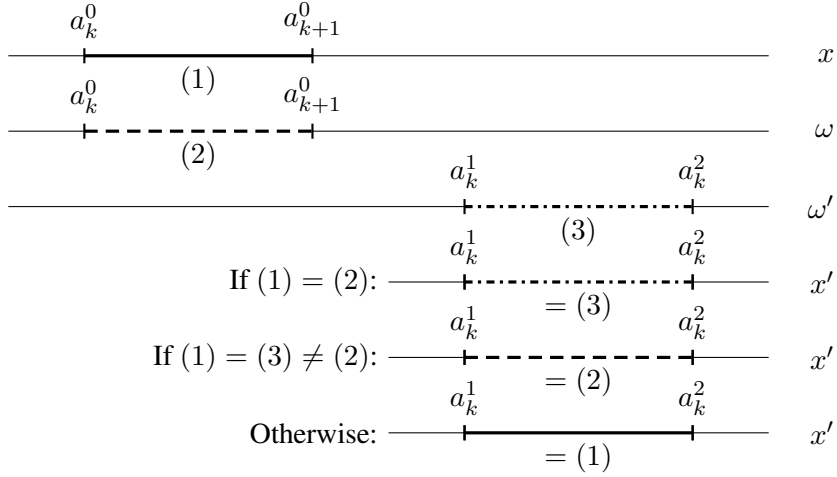

	\item[(ii)] The rest of $x'$ on $I^2_k$ is defined as follows:
	$$x'|_{I^2_k}=\left\{
	\begin{aligned}
	\omega'|_{I^2_k},\quad&x'|_{I^1_k}=\omega'|_{I^1_k}\text{ and }x'|_{I^1_{k+1}}=\omega'|_{I^1_{k+1}},\\
	W_{a,b},\quad&\text{otherwise},
	\end{aligned}
	\right.$$
	where $a=x'|_{a^2_{k}-1}$ and $b=x'|_{a^1_{k+1}}$.
	\begin{figure}
		\begin{center}
			\begin{tikzpicture}
			\foreach \x in {1}
			{
				\draw[thin,-](-4.5,\x)--(-4,\x);
				\draw[thin,-](4,\x)--(4.5,\x);
				\draw[thin,dashdotted](1,\x)--(4,\x);
				\draw[thin,dashdotted](-1,\x)--(-4,\x);
				\draw[very thick](-1,\x)--(1,\x);
				\draw(5.5,\x)node[left]{$\omega'$};
				\draw(2.5,\x)node[below]{$(2)$};
				\draw(-2.5,\x)node[below]{$(1)$};
				\draw[thick,-](1,\x-0.1)--(1,\x+0.1)node[above]{$a^1_{k+1}$};
				\draw[thick,-](4,\x-0.1)--(4,\x+0.1)node[above]{$a^2_{k+1}$};
				\draw[thick,-](-4,\x-0.1)--(-4,\x+0.1)node[above]{$a^1_{k}$};
				\draw[thick,-](-1,\x-0.1)--(-1,\x+0.1)node[above]{$a^2_{k}$};
				\draw(0,\x)node[below]{$\omega'|_{I^2_k}$};
			}
			\foreach \x in {0}
			{
				\draw[thin,-](-4.5,\x)--(-4,\x);
				\draw[thin,-](4,\x)--(4.5,\x);
				\draw[thin,dashdotted](1,\x)--(4,\x);
				\draw[thin,dashdotted](-1,\x)--(-4,\x);
				\draw(5.5,\x)node[left]{$x'$};
				\draw(-4.5,\x)node[left]{If $(1)=(3)$ and $(2)=(4)$:};
				\draw[thick,-](1,\x-0.1)--(1,\x+0.1);
				\draw[thick,-](4,\x-0.1)--(4,\x+0.1);
				\draw[thick,-](-1,\x-0.1)--(-1,\x+0.1);
				\draw[thick,-](-4,\x-0.1)--(-4,\x+0.1);
				\draw[very thick](-1,\x)--(1,\x);
				\draw(2.5,\x)node[below]{$(4)$};
				\draw(-2.5,\x)node[below]{$(3)$};
				\draw(0,\x)node[below]{$\omega'|_{I^2_k}$};
			}
			\foreach \x in {-1}
			{
				\draw[thin,-](-4.5,\x)--(-1,\x);
				\draw[thin,-](1,\x)--(4.5,\x);
				\draw(5.5,\x)node[left]{$x'$};
				\draw(-4.5,\x)node[left]{Otherwise:};
				\draw[thick,-](1,\x-0.1)--(1,\x+0.1);
				\draw[thick,-](4,\x-0.1)--(4,\x+0.1);
				\draw[thick,-](-1,\x-0.1)--(-1,\x+0.1);
				\draw[thick,-](-4,\x-0.1)--(-4,\x+0.1);
				\draw[very thick,dash pattern=on 5pt off 3pt](-1,\x)--(1,\x);
				\draw(0,\x)node[below]{$W_{a,b}$};
				\draw(1,\x)node[below]{$b$};
				\draw(-1,\x)node[below]{$a$};
			}
			\end{tikzpicture}
		\end{center}
		\caption{The symbols of $x'$ on $I^2_k$}
		\label{f:phi2}
	\end{figure}
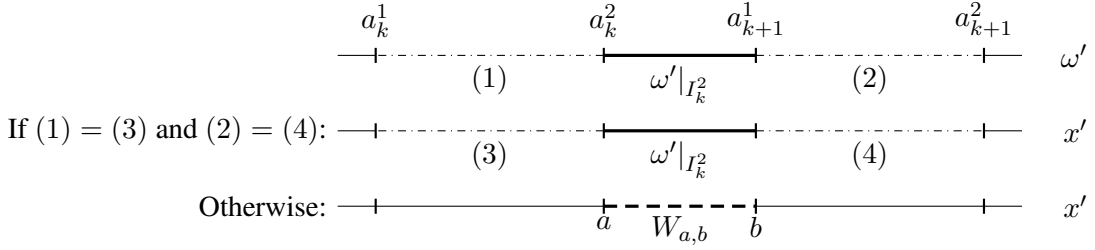
\end{itemize}
Finally, we set $\varphi^M_{\omega,\omega'}(x)=x'$.
Since $\Sigma_A$ is aperiodic and irreducible shift of finite type and $r$ such that $A^{r+1}>0$, the map $\varphi^M_{\omega,\omega'}$ is well defined.

Indeed, whether $x'|_{I^1_k}$ is equal to $\omega'|_{I^1_k}$ depends on whether $x|_{I^0_k}$ is equal to $\omega|_{I^0_k}$. If $x|_{I^0_k}\neq \omega|_{I^0_k}$, then we choose one of them as $x'|_{I^1_k}$ such that $x'|_{I^1_k}\neq\omega'|_{I^1_k}$. For an intuitive understanding, the symbols of $x'$ on $I^1_k$ are shown in Figure \ref{f:phi1}.

On the other hand, if $x|_{I^0_k\cup I^0_{k+1}}=\omega|_{I^0_k\cup I^0_{k+1}}$, we set $x'|_{I^2_k}=\omega'|_{I^2_k}$ since the fix connected word $W_{a,b}$ may be not equal to $\omega'|_{I^2_k}$. This setting will ensure that $\rho(\omega',\varphi^M_{\omega,\omega'}(x))$ can be bounded by $\rho(\omega,x)$.  For an intuitive understanding, the symbols of $x'$ on $I^2_k$ are shown in Figure \ref{f:phi2}.

Then we have the following lemma:
\begin{lemma}\label{l:phiM}
	For any $x,y\in\Sigma_A$, we have
	\begin{itemize}
		\item[(i)] $K^{-M}\rho(\omega,x)^{\frac{M+r}{M}}\le\rho(\omega',\varphi^M_{\omega,\omega'}(x))\le K^{M+2r}\rho(\omega,x)^{\frac{M+r}{M}}$;
		\item[(ii)] $K^{-M}\rho(x,y)^{\frac{M+r}{M}}\le\rho(\varphi^M_{\omega,\omega'}(x),\varphi^M_{\omega,\omega'}(y))\le K^{M+2r}\rho(x,y)^{\frac{M+r}{M}}$.
	\end{itemize}
\end{lemma}

\begin{proof}
	It is obvious that if $x=\omega$ then $\varphi^M_{\omega,\omega'}(x)=\omega'$.
	So we assume that $x\neq \omega$.
	Let $x'=\varphi^M_{\omega,\omega'}(x)$, $p=\min\{i\in\N:\omega_i\neq x_i\}$, and $k=\lfloor \frac{p}{M}\rfloor$ be the largest integer not greater than $\frac{p}{M}$. If $k=0$, then $\omega'|_{I^1_0}\neq x'|_{I^1_0}$, that is, $$K^{-M}\rho(\omega,x)^{\frac{M+r}{M}}\le K^{-M}\le \rho(\omega',x')\le 1\le K^{M+2r}\rho(\omega,x)^{\frac{M+r}{M}}.$$
	
	If $k\ge 1$, we have $\min\{i\in\N:\omega'_i\neq x'_i\}\in I^2_{k-1}\cup I^1_k$.
	Thus,
	$$K^{-M}\rho(\omega,x)^{\frac{M+r}{M}}\le\rho(\omega',x')\le K^{M+2r}\rho(\omega,x)^{\frac{M+r}{M}},$$
	which proves (i).
	
	The proof of (ii) is similar to (i).
\end{proof}

By (i) of Lemma \ref{l:phiM}, for $M\ge r$, we have
$$K^{-M}\rho(\omega,x)^2\le\rho(\omega',\varphi^M_{\omega,\omega'}(x))\le K^{M+2r}\rho(\omega,x).$$
Since $\sigma^{M+r}\varphi^M_{\omega,\omega'}(x)=\varphi^M_{\sigma^M\omega,\sigma^{M+r}\omega'}(\sigma^Mx)$, then
$$\rho(\sigma^{Mi}\omega,\sigma^{Mi}x)^2\le \rho(\sigma^{(M+r)i}\omega,\sigma^{(M+r)i}\varphi^M_{\omega,\omega'}(x))\le \rho(\sigma^{Mi}\omega,\sigma^{Mi}x).$$
By the inequality $\frac1n\sum_{i=0}^{n-1}a_i^2\ge \left(\frac1n\sum_{i=0}^{n-1}a_i\right)^2$, we have
$$\varphi^M_{\omega,\omega'}(\mathrm{ML}_{\omega}(\Sigma_A,\sigma^M))\subset \mathrm{ML}_{\omega'}(\Sigma_A,\sigma^{M+r}).$$
Noticing that $\mathrm{ML}_\omega(\Sigma_A,\sigma^n)=\mathrm{ML}_\omega(\Sigma_A,\sigma)$ for each $n\in\N$, we have $$\varphi^M_{\omega,\omega'}(\mathrm{ML}_\omega(\Sigma_A,\sigma))\subset\mathrm{ML}_{\omega'}(\Sigma_A,\sigma).$$
And by (ii) of Lemma \ref{l:phiM}, $\varphi^M_{\omega,\omega'}$ is bi-Lipschitz. Then we have
$$\frac{M}{M+r}\dim_H(\mathrm{ML}_{\omega}(\Sigma_A,\sigma))=\dim_H(\varphi^M_{\omega,\omega'}(\mathrm{ML}_{\omega}(\Sigma_A,\sigma)))\le \dim_H(\mathrm{ML}_{\omega'}(\Sigma_A,\sigma)).$$
Let $M\To\infty$, $\dim_H(\mathrm{ML}_{\omega}(\Sigma_A,\sigma))\le\dim_H(\mathrm{ML}_{\omega'}(\Sigma_A,\sigma))$. Thus for any  $\omega,\omega'\in\Sigma_A$, $$\dim_H(\mathrm{ML}_{\omega}(\Sigma_A,\sigma))=\dim_H(\mathrm{ML}_{\omega'}(\Sigma_A,\sigma))$$
and
$$\dim_P(\mathrm{ML}_{\omega}(\Sigma_A,\sigma))=\dim_P(\mathrm{ML}_{\omega'}(\Sigma_A,\sigma))$$
where the formula of $\dim_P$ is similar to the case of $\dim_H$.
Since there is an ergodic quasi-Bernoulli measure on $\Sigma_A$, then by Theorem \ref{t:app}, it ends the proof.
\end{proof}

\begin{remark}\label{r:LY}
	By \cite[Theorem 6.14]{YCL19}, it is proved that $\mathcal{H}^2(\mathrm{LY}(\Sigma,\sigma))=1$, where $\mathrm{LY}(\Sigma,\sigma)$ is the set of all Li--Yorke pairs of $\Sigma$. Since $\mathcal{H}^2=\mathcal{H}^1\times\mathcal{H}^1$, we have $\mathcal{H}^1(\mathrm{LY}_\omega(\Sigma,\sigma))=1$ for $\mathcal{H}^1$-almost everywhere $\omega\in\Sigma$, where $\mathrm{LY}_\omega(\Sigma,\sigma)$ is the set of all points $x\in\Sigma$ such that $(\omega,x)$ is a Li--Yorke pair.
	Then by $\varphi_{\omega,\omega'}$ in the simple case of the proof of Theorem \ref{t:app1.5}, it shows that $\mathcal{H}^1(\mathrm{LY}_\omega(\Sigma,\sigma))=1$ for each $\omega\in\Sigma$.
\end{remark}

Finally, we can prove Theorem \ref{t:ml}.

\begin{proof}[Proof of Theorem \ref{t:ml}]
	For an irreducible shift of finite type $\Sigma_A$, by (\cite[Section 4.5, Proposition 4.5.6]{LM95}), we have
	$$\Sigma_A=\biguplus_{i=0}^{p-1}\Sigma_i$$
	where $\Sigma_i$ is the union of some cylinders whose length is one, $\sigma(\Sigma_i)=\Sigma_{(i+1)\mathrm{mod}\,p}$ and $(\Sigma_i,\sigma^p)$ is topologically conjugate to an aperiodic and irreducible shift of finite type $(X_i,\sigma)$.
	Then $h^B(\Sigma_{i})\ge h^B(\Sigma_{(i+1)\mathrm{mod}\,p})$, and by $h(\Sigma_A)=\sup_{0\le i\le p-1}h^B(\Sigma_i)$, we have $h^B(\Sigma_i)=h(\Sigma_A)$.
	
	Fix any $i$ and $x\in\Sigma_i$. We claim that $\mathrm{ML}_x(\Sigma_i,\sigma^p)=\mathrm{ML}_x(\Sigma_A,\sigma)$.
	It only needs to prove that $\Sigma_i\supset\mathrm{ML}_x(\Sigma_A,\sigma^p)$.
	Fix any $y\in\mathrm{ML}_x(\Sigma_A,\sigma^p)$. If $y\notin\Sigma_i$, we have $\sigma^{pj}y\notin\Sigma_i$, which implies that $\rho(\sigma^{pj}x,\sigma^{pj}y)=1$ for any $j\in\N$, contradicted to $y\in\mathrm{ML}_x(\Sigma_A,\sigma^p)$.
	
	Let $\phi_i:\Sigma_i\To X_i$ be the conjugation from $(\Sigma_i,\sigma^p)$ to $(X_i,\sigma)$.
	Then  $h^B(\mathrm{ML}_x(\Sigma_i,\sigma^p))=\frac1ph^B(\mathrm{ML}_{\phi_i(x)}(X_i,\sigma))$.
	By Theorem \ref{t:app1.5}, we have $$h^B(\mathrm{ML}_{\phi_i(x)}(X_i,\sigma))=0,$$
	which implies that
	$$h^B(\mathrm{ML}_x(\Sigma_A,\sigma))=0.$$
	
	Similarly, it can be proved that $h^P(\mathrm{ML}_x(\Sigma_A,\sigma))=h(\Sigma_A)$ for each $x\in\Sigma_A$ by $h^P(\Sigma_i)=\frac1ph(X_i)$.
\end{proof}

\begin{remark}
	For any cylinder $W$ in $\Sigma_A$ with $[W]\cap\Sigma_i\neq\emptyset$, since $(\Sigma_i,\sigma^p)$ is topologically conjugate to an aperiodic and irreducible shift of finite type $(X_i,\sigma)$, there exists $m\in\mathbb{N}$ such that $\sigma^{mp}([W])=\Sigma_i$.
	Then we have $\sigma^{mp}([W]\cap\mathrm{ML}_x(\Sigma_A,\sigma))=\mathrm{ML}_{\sigma^{mp}x}(\Sigma_A,\sigma)$, and by $\sigma^{mp}$ is finite-to-one, $h^P([W]\cap\mathrm{ML}_x(\Sigma_A,\sigma))=h(\Sigma_A)$.
\end{remark}

\section{Proof of Theorem \ref{t:ly}}\label{s:8}

In this section, we will prove Theorem \ref{t:ly}. Similar to Section \ref{s:7}, the proof is divided into two parts.
First, we prove that Theorem \ref{t:ly} holds for aperiodic and irreducible shifts of finite type. The proof also construct a map transferring almost everywhere to everywhere.

\begin{theorem}\label{t:ly2}
	Assume that $\Sigma_A$ is an aperiodic and irreducible shift of finite type. Then for any $\omega\in\Sigma_A$,
	$$h^B(\mathrm{LY}_\omega(\Sigma_A,\sigma))=h(\Sigma_A),$$
	that is,
	$$\dim_H(\mathrm{LY}_\omega(\Sigma_A,\sigma))=\dim_H(\Sigma_A).$$
\end{theorem}

\begin{proof}
	Fix any $\epsilon>0$. Then for large enough $M$, we have
	$$\#\Sigma_{A,M}\ge e^{M(h(\Sigma_A)-\epsilon)}.$$
	Let $L=\#\Sigma_{A,M}$ and $\Sigma_{A,M}=\{W_0,W_1,\dots,W_{L-1}\}.$
	Recall that there exists $r>0$ such that $A^{r+1}>0$. Then for any two symbols $a,b\in\mathcal{A}$, we fix a connected word $W_{a,b}\in\Sigma_{A,r}$ such that $aW_{a,b}b\in\Sigma_{A,r+2}$.
	Let $I^1_k=[k(M+r),k(M+r)+M)$ and $I^2_k=[k(M+r)+M,(k+1)(M+r))$.
	Then for any $\omega\in\Sigma_A$, we define $\varphi^M_\omega:\{0,1,\dots,L-1\}^\N\To\Sigma_A$ by two steps:
	\begin{itemize}
		\item[(1)] for each $k\in\N$, define the symbols of $\varphi^M_\omega(z)$ on $I^1_k$ by
		$$\varphi^M_\omega(z)|_{I^1_k}=W_{(i_k+z_k)\mathrm{ mod }L},$$
		where $i_k$ is the unique number such that $W_{i_k}=\omega|_{I^1_k}$;
		\item[(2)] for each $k\in\N$, define the symbols of $\varphi^M_\omega(z)$ on $I^2_k$ by
		$$\varphi^M_\omega(z)|_{I^2_k}=\left\{
		\begin{aligned}
		&x|_{I^2_k},\quad\text{if }z_k=z_{k+1}=0,\\
		&W_{a,b},\quad\text{otherwise,}
		\end{aligned}
		\right.$$
		where $a=\varphi^M_\omega(z)|_{k(M+r)+M-1}$ and $b=\varphi^M_\omega(z)|_{(k+1)(M+r)}$.
	\end{itemize}
	
	It is no hard to see that $\varphi^M_\omega$ is well-defined and continuous.
	Similar to Lemma \ref{l:phiM}, we have the following lemma:
	\begin{lemma}\label{l:phiM2}
		For any $\omega\in\Sigma_A$ and $z,z'\in\{0,1,\dots,L-1\}^\N$, we have
		\begin{itemize}
			\item[(i)] $K^{-M}\rho(\underline{0},z)^{(M+r)\frac{\log K}{\log L}}\le \rho(\omega,\varphi^M_\omega(z))\le K^r\rho(\underline{0},z)^{(M+r)\frac{\log K}{\log L}}$, where $\underline{0}=000\cdots\in\{0,1,\dots,L-1\}^\N$;
			\item[(ii)] $K^{-M}\rho(z,z')^{(M+r)\frac{\log K}{\log L}}\le \rho(\varphi^M_\omega(z),\varphi^M_\omega(z'))\le K^r\rho(z,z')^{(M+r)\frac{\log K}{\log L}}$.
		\end{itemize}
	\end{lemma}
	
	Since $\varphi_{\sigma^{M+r}\omega}^M\circ\sigma=\sigma^{M+r}\circ\varphi^M_{\omega}$, by (i) of Lemma \ref{l:phiM2}, we have
	$$\varphi^M_\omega(\mathrm{LY}_{\underline{0}}(\{0,1,\dots,L-1\}^\N,\sigma))\subset \mathrm{LY}_{\omega}(\Sigma_A,\sigma^{M+r}).$$
	
	Noticing that $\mathrm{LY}_{\omega}(\Sigma_A,\sigma^n)=\mathrm{LY}_{\omega}(\Sigma_A,\sigma)$ for each $n\in\N$, we have $$\varphi^M_\omega(\mathrm{LY}_{\underline{0}}(\{0,1,\dots,L-1\}^\N,\sigma))\subset \mathrm{LY}_{\omega}(\Sigma_A,\sigma).$$
	And by (ii) of Lemma \ref{l:phiM2}, $\varphi^M_\omega$ is bi-Lipschitz. Then we have
	$$
	\begin{aligned}
	\dim_H(\mathrm{LY}_{\omega}(\Sigma_A,\sigma))\ge&\dim_H(\varphi^M_{\omega}(\mathrm{LY}_{\underline{0}}(\{0,1,\dots,L-1\}^\N,\sigma)))\\
	=&\frac{\log L}{(M+r)\log K}\dim_H(\mathrm{LY}_{\underline{0}}(\{0,1,\dots,L-1\}^\N,\sigma)).
	\end{aligned}
	$$
	By Remark \ref{r:LY}, we have $\dim_H(\mathrm{LY}_{\underline{0}}(\{0,1,\dots,L-1\}^\N,\sigma))=1$.
	Then
	$$
	\dim_H(\mathrm{LY}_{\omega}(\Sigma_A,\sigma))\ge\frac{\log L}{(M+r)\log K}
	\ge\frac{M(h(\Sigma_A)-\epsilon)}{(M+r)\log K},
	$$
	where the last inequality holds since $L=\#\Sigma_{A,M}\ge e^{M(h(\Sigma_A)-\epsilon)}$.
	Let $M\To\infty$ and then $\epsilon\To0$, we have $\dim_H(\mathrm{LY}_{\omega}(\Sigma_A,\sigma))\ge\frac{h(\Sigma_A)}{\log K}=\dim_H(\Sigma_A)$. Thus for any  $\omega\in\Sigma_A$, $$\dim_H(\mathrm{LY}_{\omega}(\Sigma_A,\sigma))=\dim_H(\Sigma_A),$$
	which ends the proof.
\end{proof}

Now we can prove Theorem \ref{t:ly}.

\begin{proof}[Proof of Theorem \ref{t:ly}]
	For an irreducible shift of finite type $\Sigma_A$, by (\cite[Section 4.5, Proposition 4.5.6]{LM95}), we have
	$$\Sigma_A=\biguplus_{i=0}^{p-1}\Sigma_i$$
	where $\Sigma_i$ is the union of some cylinders of length one, $\sigma(\Sigma_i)=\Sigma_{(i+1)\mathrm{mod}\,p}$ and $(\Sigma_i,\sigma^p)$ is topologically conjugate to an aperiodic and irreducible shift of finite type $(X_i,\sigma)$.
	Then $h^B(\Sigma_{i})\ge h^B(\Sigma_{(i+1)\mathrm{mod}\,p})$, and by $h(\Sigma_A)=\sup_{0\le i\le p-1}h^B(\Sigma_i)$, we have $h^B(\Sigma_i)=h(\Sigma_A)$.
	
	Fix any $i$ and $x\in\Sigma_i$. We claim that $\mathrm{LY}_x(\Sigma_i,\sigma^p)=\mathrm{LY}_x(\Sigma_A,\sigma)$.
	It only needs to prove that $\Sigma_i\supset\mathrm{LY}_x(\Sigma_A,\sigma)$.
	Fix any $y\in\mathrm{LY}_x(\Sigma_A,\sigma)$. If $y\notin\Sigma_i$, we have $\sigma^{pj}y\notin\Sigma_i$, which implies that $x_{pj}\neq y_{pj}$ for any $j\in\N$.
	Then for any $n\in\N$, let $n=pj-q$ for some $j\in\N$ and $0\le q<p$. Since $x_{pj}\neq y_{pj}$, we have $\rho(\sigma^nx,\sigma^ny)\ge K^{-q}\ge K^{-p+1}$, which is contradicted to $y\in\mathrm{LY}_x(\Sigma_A,\sigma)$.
	
	Let $\phi_i:\Sigma_i\To X_i$ be the conjugation from $(\Sigma_i,\sigma^p)$ to $(X_i,\sigma)$.
	Then $h^B(\Sigma_i)=\frac1ph(X_i)$ and $h^B(\mathrm{LY}_x(\Sigma_i,\sigma^p))=\frac1ph^B(\mathrm{LY}_{\phi_i(x)}(X_i,\sigma))$.
	By Theorem \ref{t:ly2}, we have $$h^B(\mathrm{LY}_{\phi_i(x)}(X_i,\sigma))=h(X_i),$$
	which implies that
	$$h^B(\mathrm{LY}_x(\Sigma_A,\sigma))=h^B(\Sigma_i)=h(\Sigma_A).$$
	
\end{proof}

\begin{remark}
	For any cylinder $W$ in $\Sigma_A$ with $[W]\cap\Sigma_i\neq\emptyset$, since $(\Sigma_i,\sigma^p)$ is topologically conjugated to an aperiodic and irreducible shift of finite type $(X_i,\sigma)$, there exists $m\in\mathbb{N}$ such that $\sigma^{mp}([W])=\Sigma_i$.
	Then we have $\sigma^{mp}([W]\cap\mathrm{LY}_x(\Sigma_A,\sigma))=\mathrm{LY}_{\sigma^{mp}x}(\Sigma_A,\sigma)$, and by $\sigma^{mp}$ is finite-to-one, $h^B([W]\cap\mathrm{LY}_x(\Sigma_A,\sigma))=h(\Sigma_A)$.
\end{remark}

\section*{Acknowledgements}
	The work was supported by the
	National Natural Science Foundation of China (Nos.12071222 and 11971236), China Postdoctoral Science Foundation (No.2016M591873),
	and China Postdoctoral Science Special Foundation (No.2017T100384). The work was also funded by the Priority Academic Program Development of Jiangsu Higher Education Institutions.  We would like to express our gratitude to Tianyuan Mathematical Center in Southwest China(11826102), Sichuan University and Southwest Jiaotong University for their support and hospitality.

\end{document}